\documentclass[11pt,reqno, namelimits, sumlimits,a4paper]{amsart} 

\usepackage[margin=1in]{geometry}
\usepackage{times}

\usepackage{latexsym}
\usepackage{amssymb,amsfonts,amsthm}
\usepackage{comment}
\usepackage{enumerate}
\usepackage{color}
\usepackage{esint} 
\usepackage{graphicx}
\usepackage[title]{appendix}

\DeclareMathAlphabet{\mathcal}{OMS}{cmsy}{m}{n}

\theoremstyle{plain}

\newtheorem{theorem}{Theorem}[section]

\newtheorem{definition}[theorem]{Definition}
\newtheorem{lemma}[theorem]{Lemma}

\newtheorem{corollary}[theorem]{Corollary}
\newtheorem{pro}[theorem]{Proposition}

\theoremstyle{definition}
\newtheorem{remark}[theorem]{Remark}

\newcommand{\R}{\mathbb{R}}

\newcommand{\N}{\mathbb{N}}

\newcommand{\bP}{\mathbb{P}}
\newcommand{\bQ}{\mathbb{Q}}

\newcommand{\cN}{\mathcal N}
\newcommand{\cF}{\mathcal F}

\newcommand{\cO}{\mathcal O}

\newcommand{\cR}{\mathcal R}

\newcommand{\p}{\partial}
\newcommand{\la}{\langle}
\newcommand{\ra}{\rangle}

\newcommand{\norm}[1]{\lVert #1 \rVert}
\renewcommand{\:}{\colon}
\newcommand{\upto}{\uparrow}
\newcommand{\dto}{\downarrow}
\newcommand{\wto}{\rightharpoonup} 
\newcommand{\wstar}{\overset{\ast}{\rightharpoonup}}

\newcommand{\para}{{\rm par}}

\newcommand{\weak}{{\rm weak}}
\let\div\relax
\DeclareMathOperator{\div}{div}
\DeclareMathOperator{\dist}{dist}

\DeclareMathOperator{\tr}{tr}
\let\tilde\relas
\newcommand{\tilde}[1]{\widetilde{#1}}
\DeclareMathOperator*{\esssup}{ess\,sup}

\newcommand{\avg}{{\rm avg}}

\newcommand{\BMO}{{\rm BMO}}

\numberwithin{equation}{section}

\title[Localised conditions for singularity formation in NSE with curved boundary]{Localised necessary conditions for singularity formation in the Navier-Stokes equations with curved boundary}
\author{Dallas Albritton$^\dag$}
\thanks{$\dag$: School of Mathematics, University of Minnesota, 206 Church St. SE, Minneapolis, MN 55455 \\Email address: \texttt{albri050@umn.edu}. (Corresponding author)}
\author{Tobias Barker$^\ast$}
\thanks{$\ast$: DMA, {\'E}cole Normale Sup{\'e}rieure, CNRS, PSL Research University, 75005 Paris \\Email address: \texttt{tobiasbarker5@gmail.com}}
\date{\today}

\begin{document}

\begin{abstract}
We generalize two results in the Navier-Stokes regularity theory whose proofs rely on `zooming in' on a presumed singularity to the local setting near a curved portion $\Gamma \subset \p\Omega$ of the boundary. Suppose that $u$ is a boundary suitable weak solution with singularity $z^* = (x^*,T^*)$, where $x^* \in \Omega \cup \Gamma$. Then, under weak background assumptions, the $L_3$ norm of $u$ tends to infinity in every ball centered at $x^*$: 
\begin{equation*}
\lim_{t \to T^*_-} \norm{u(\cdot, t)}_{L_{3}\left(\Omega \cap B(x^*,r)\right)} = \infty \quad \forall r > 0.
\end{equation*}
Additionally, $u$ generates a non-trivial `mild bounded ancient solution' in $\R^3$ or $\R^3_+$ through a rescaling procedure that `zooms in' on the singularity. 
Our proofs rely on a truncation procedure for boundary suitable weak solutions. The former result is based on energy estimates for $L_3$ initial data and a Liouville theorem. For the latter result, we apply perturbation theory for $L_\infty$ initial data based on linear estimates due to K.~Abe and Y.~Giga.
\end{abstract}
\maketitle

\textit{Keywords}:
Navier-Stokes equations, ancient solutions, Liouville theorems

\tableofcontents

\section{Introduction}
In this paper, we investigate the local regularity of Navier-Stokes solutions near a curved portion of the boundary. We are particularly interested in aspects of the regularity theory whose current proofs rely on `zooming in' on a presumed singularity and Liouville theorems, {\`a} la De Giorgi~\cite{degiorgi65}. This technique was famously used by Escauriaza, Seregin, and {\v S}ver{\'a}k~\cite{escauriazasereginsverak} to solve the difficult endpoint $s=3$ of the Ladyzhenskaya-Prodi-Serrin criterion:
\begin{equation}
	\label{eq:serrincond}
	v \in L_{s,l}(Q), \quad  \frac{3}{s} + \frac{2}{l} \leq 1.
\end{equation}
Namely, Escauriaza \emph{et al.} proved that any suitable weak solution in a parabolic ball $Q$ and belonging to the critical space $L_{3,\infty}(Q)$ is H{\"o}lder-continuous in each smaller parabolic ball $Q'$.

The presence of solid boundaries introduces new difficulties due to the pressure and the no-slip condition. In the interior, the pressure is often decomposed into
\begin{equation}
	\tilde{p} = (-\Delta)^{-1} \div \div (\varphi v \otimes v) \approx |v|^2
\end{equation} and a harmonic function $h$. Against the boundary, the above decomposition is not as effective, since $h$ does not evidently satisfy a useful boundary condition.
For this purpose, Seregin~\cite{seregincknflatboundary} introduced a new decomposition of the pressure that opened the door to new results with boundary. Partial regularity of suitable weak solutions up to the boundary was proven by Seregin, Shilkin, and Solonnikov~\cite{seregincknflatboundary,sereginshilkinsolonnikovboundarypartialreg2014}, and the regularity of $L_{3,\infty}$ solutions near the boundary was proven by Seregin, Shilkin, and Mikhailov~\cite{sereginl3inftyflatboundary,mikhailovshilkincurvedboundary2006} (see the survey~\cite{sereginshilkinsurvey}). The analogous theory in dimensions $n \geq 4$ was recently developed by Dong and Wang~\cite{dongguboundarypartial,dongwangboundarylebesgue}.




A further refinement of the regularity theory is due to Seregin~\cite{sereginl3}:
\begin{equation}
	\label{sereginproved}
	\lim_{t \to T^*_-} \norm{v(\cdot,t)}_{L_3(\R^3)} = \infty,
\end{equation}
where $T^* \in (0,\infty)$ is the maximal time of smoothness of a (supposedly singular) weak Leray-Hopf solution with initial data in $C^\infty_0$. A key observation is that, to control the local energy of a sequence of rescaled solutions, it is enough to control the original solution in $L_3$ along a \emph{sequence} of times. This idea was adapted to the half-space $\R^3_+$ by Seregin and the second author in~\cite{barkerser16blowup} and abstracted in~\cite{sereginsverakweaksols}, and the methods therein have led to further interesting developments~\cite{barkersereginsverakstability,globalweakbesov}. However, whereas the main theorem in Escauriaza \emph{et al.}~\cite{escauriazasereginsverak} is formulated locally and may be applied to any neighborhood of a singularity, the methods in the above works rely on a global approach. This begs the following question:
\begin{center}
\textbf{Q}. Does the critical $L_3$ norm tend to infinity in \emph{every neighborhood} of a presumed singularity?
\end{center}

We answer this question in the affirmative below:

\begin{theorem}
\label{cor:l3}
Let $\Omega \subset \R^3$ be a bounded $C^2$ domain and $\Gamma \subset \p\Omega$ a relatively open subset of the boundary. Let $x^* \in \Omega\cup \Gamma$ and $\Omega_{x^*,R} = \Omega \cap B(R)$ whenever $R>0$. Let $T^* > 0$.

Let $v$ be a boundary suitable weak solution in $\Omega \times ]0,T^*[$ vanishing on $\Gamma$ and satisfying
\begin{equation}\label{vboundedfinalmomentoftime}
 v\in L_{\infty}(\Omega \times ]0,t[) \; \text{ for all } 0<t<T^*.
 \end{equation}
If
\begin{equation}
	z^* = (x^*,T^*) \text{ is a singular point of } v,
\end{equation}
then, for all $R > 0$,
\begin{equation}\label{localL3}
\lim_{t \to T^*_-} \norm{v(\cdot, t)}_{L_{3}(\Omega_{x^*,R})} = \infty.
\end{equation}
\end{theorem}

Theorem~\ref{cor:l3} is new even in the absence of boundary. In Section~\ref{sec:l3}, we present Theorem~\ref{thm:l3}, which adapts Theorem~\ref{cor:l3} to the $L_3^\weak$ setting in a quantitative way.

It is interesting to compare~\eqref{localL3} with recent `concentration' results~\cite{neustupa2013,liozawawang,prangeliouville,barker2018localized} in the Navier-Stokes literature. For example,~\eqref{localL3} generalizes Neustupa's concentration result~\cite{neustupa2013} at fixed spatial scales:
\begin{equation}
	\label{neustupa}
	\liminf_{t \to T^*_-} \norm{v(\cdot,t)}_{L_3(B(x^*,R))} \geq \epsilon,
\end{equation}
for all $R > 0$ and an absolute constant $\epsilon > 0$. In~\cite{liozawawang,prangeliouville}, it is demonstrated that
one may take $R \sim \sqrt{T^*-t}$ in~\eqref{neustupa} in exchange for $\sup_{x^* \in \R^3}$ within the $\liminf_{t \to T^*_-}$. The $\sup_{x^* \in \R^3}$ condition was removed in the recent work~\cite{barker2018localized} by Prange and the second author, which appeared after the present work. It would be interesting to know if~\eqref{localL3} remains true with $R \sim \sqrt{T^*-t}$. \\

Our analysis hinges on a localisation procedure which was introduced in~\cite{neustupapenel1999} by Neustupa and Penel in the context of interior one-component regularity criteria. Truncating a solution of the Stokes equations by a smooth cut-off $\Phi$ introduces a forcing term $f$ and typically requires solving $\div w = -\nabla \Phi \cdot v$ to correct the non-zero divergence. However, to analyze the regularity of a truncated Navier-Stokes solution (which, \emph{a priori}, may be singular), we desire $f$ and $w$ to be subcritical. Whereas $w$ gains a derivative and is less dangerous, $f$ contains the problematic term $(\Phi^2 - \Phi) v \cdot \nabla v$.
 This term appears to be supercritical, since in Theorem~\ref{cor:l3}, $v$ is only controlled in $L_3$ along a sequence of times. To circumvent this difficulty, one may rely on partial regularity in three dimensions, which guarantees that there is a parabolic annulus on which $v$ is bounded. Truncating on this annulus ensures that $f$ is subcritical and disappears upon `zooming in'. In this paper, we expend additional effort to control the constant in Bogovskii's operator with curved boundary. The details are contained in Proposition~\ref{pro:truncation}. We expect this procedure to be useful in localising many other regularity criteria, e.g., the vorticity alignment criterion in~\cite{gigahsumaekawaplaner}. This was recently explored in~\cite{barker2019scale} by Prange and the second author, which appeared after the present work.\footnote{See also~\cite[Remark~12.3]{taolocalizationcompactness} and~\cite{kukavicarusinzianeJMFM2017} for other applications.} However, the truncation does not appear to work well in higher dimensions unless one assumes additional conditions guaranteeing $\mathcal{H}^1(S) = 0$.

Once the solution has been truncated around the singularity, we follow the general scheme in~\cite{sereginl3,barkerser16blowup}. That is, upon rescaling, we obtain a sequence of solutions on a growing sequence of domains $\Omega_k \times ]0,1[$. These solutions are controlled in the energy space using a Calder{\'o}n-type splitting~\cite{calderon90,Jiasver2013,albrittonblowupcriteria}. 
If the singularity is on the boundary, we flatten the boundary and pass to a singular `blow-up limit' in $\R^3_+$ that vanishes identically at $t=1$. Backward uniqueness arguments from~\cite{barkerser16blowup} give the final contradiction. \\

In the second half, we are concerned with \emph{mild bounded ancient solutions}, which arise naturally as `blow-up limits' of singular Navier-Stokes solutions. By definition, these solutions are bounded (in fact, smooth) and satisfy the integral formulation of the Navier-Stokes equations on $\R^3$ or $\R^3_+$ for all backward times, which excludes certain `parasitic solutions' driven by the pressure. The \emph{Liouville conjecture} formulated by G. Koch, Seregin, {\v S}ver{\'a}k, and Nadirashvili in~\cite{kochnadirashvili} is the following:
\begin{center}
\textbf{Conjecture}. Mild bounded ancient solutions in $\R^3$ are constants.
\end{center} 
While the conjecture is far from settled, it is known to hold in special circumstances, e.g., in dimension two, in the axisymmetric setting without swirl~\cite{kochnadirashvili}, and in the periodic-in-$z$ setting~\cite{carrillo2018decay,lei2019ancient}. Even more is unknown regarding mild bounded ancient solutions in the half-space~\cite{sereginsverakrescalinghalfspace,barkersereginmbas,sereginliouvillehalfspacel2infty} (see~\cite{sereginshilkinliouvillesurvey} for a recent survey).

In recent work~\cite{albrittonbarkerlocalregI}, we demonstrated that Type~I singularity formation for suitable weak solutions is \emph{equivalent} to the existence of a non-trivial (that is, not identically zero) mild bounded ancient solution in $\R^3$ satisfying a Type~I decay condition. In this paper, we focus on the forward direction without the Type~I assumption. Specifically, we demonstrate that `zooming in' on a singular boundary suitable weak solution generates a non-trivial mild bounded ancient solution under quite general assumptions. 

\begin{theorem}[Existence of mild bounded ancient solutions]
\label{thm:mbas}
Let $\Omega \subset \R^3$ be a bounded $C^3$ domain and $\Gamma \subset \p\Omega$ a relatively open subset of the boundary. Let $T > 0$.

Let $v$ be a boundary suitable weak solution in $Q_T = \Omega \times ]0,T[$ vanishing on $\Gamma$.
	\begin{itemize}
	\item If $v$ has an interior singularity, then there exists a non-trivial mild bounded ancient solution in $\R^3$ that arises as a blow-up limit of $v$.
	\item If $v$ has a boundary singularity on $\Gamma$, then there exists a non-trivial mild bounded ancient solution in $\R^3$ or $\R^3_+$ that arises as a blow-up limit of $v$.
	\end{itemize}
\end{theorem}
That is, the mild bounded ancient solution in question is the solution obtained from the rescaling procedure in the proof of Theorem~\ref{thm:mbas}. We hope that this result and those in~\cite{albrittonbarkerlocalregI} help clarify the role of mild bounded ancient solutions in the regularity theory of the Navier-Stokes equations.

The interior case of Theorem~\ref{thm:mbas} was known, see~\cite{sereginsverakaxisymmetric} and~\cite{sereginsverakhandbook}. Regarding boundaries, Seregin and {\v S}ver{\'a}k demonstrated in~\cite{sereginsverakrescalinghalfspace} that half-space singular solutions generate mild bounded ancient solutions in $\R^3$ or $\R^3_+$. Note that, \emph{a priori}, either may occur when the domain is $\R^3_+$, depending on the rate at which the velocity growns near the boundary. The analysis in~\cite{sereginsverakrescalinghalfspace} relies on the explicit kernel representation due to Solonnikov~\cite{Sol1977} for mild solutions in the half-space. 
The most difficult part of their analysis is to obtain decay estimates for $\nabla p$ as $x_3 \to \infty$ in order to rule out `parasitic solutions' and conclude that the blow-up limit is indeed mild.

Since kernel estimates are unavailable (or unwieldy) for more general domains, we rely on a perhaps more conceptual approach, based on tools developed by K. Abe and Y. Giga in~\cite{AbeGiga,abestokesflow,AbeNSEBounded}. Once we have truncated the solution, the rescaling procedure in~\cite{sereginsverakrescalinghalfspace} gives a sequence of solutions $v^{(k)}$ ($|v^{(k)}| \leq 1$) on a growing sequence of domains $\Omega_k$ expanding to $\R^3$ or $\R^3_+$. It is vital to only use estimates which do not degenerate as the domains grow. Therefore, it is natural to apply the perturbation theory for $L_\infty$ mild solutions in bounded domains in Abe's paper~\cite{AbeNSEBounded}. We also control a correction coming from the non-zero forcing $f^{(k)}$. This yields $C^\alpha_\para$ estimates for $v^{(k)}$ and enough compactness to show that the blow-up limit is non-trivial. To complete the proof, we use scaling-invariant pressure estimates weighted by the distance to the boundary, as in~\cite{AbeGiga,abestokesflow}, that are similar to (but not exactly)
\begin{equation}
	t^{\frac{1}{2}} \sup_{x \in \Omega} {\rm dist}(x,\p\Omega) |\nabla \pi(\cdot,t)| \leq C(\Omega) \norm{u_0}_{L_\infty(\Omega)}.
\end{equation}
 These estimates are used to show that $\nabla p \to 0$ as $x_3 \to \infty$ for the blow-up limit, thereby ruling out parasitic solutions. 

 We expect that Theorem~\ref{thm:mbas} remains valid when $\Omega$ is only assumed to be a bounded $C^2$ domain (see Remark~\ref{rmk:regularityremark}). Throughout the paper, it is only essential that $\Gamma \subset \p\Omega$ is sufficiently regular, since the domain $\Omega$ may be modified.

 After the present work appeared on the arXiv, subsequently and independently, Giga \emph{et al}.~\cite{giga2019continuous} used a `zooming in' procedure  in bounded domains like the one in Section~\ref{sec:mbas} (although without truncation procedure) to obtain a vorticity alignment criterion. See also work by the second author and C. Prange~\cite{barker2019scale}. \\

For the reader's convenience and to make the paper more self-contained, we include an appendix. Appendix~\ref{sec:persistenceofsing} discusses boundary suitable weak solutions of the \emph{flattened} NSE (utilized in~\cite{sereginshilkinsolonnikovboundarypartialreg2014,mikhailovshilkincurvedboundary2006}) and proves the `persistence of singularities' (Proposition~\ref{stabilitysingularpointshalfspace}) for zooming in on a singularity against a curved portion of boundary. This result is new for curved boundaries, although it follows from known techniques in~\cite{sereginl3inftyflatboundary,Barkerthesis}. Appendix~\ref{sec:parabolicsobolev} recalls a parabolic Sobolev embedding theorem into H{\"o}lder spaces (in particular, a scaling-invariant version in the case $u|_{\p'Q_T} = 0$). Finally, Appendix~\ref{sec:neumannprob} collects known \emph{a priori} estimates, weighted by the distance to the boundary, for harmonic functions with divergence-form Neumann data. 

\subsubsection*{Notation}

For $x = (x',x_3) \in \R^{2+1}$, $t \in \R$, $z=(x,t)$, and $R>0$, we define
\begin{equation}
	B(x,R) = \{ y \in \R^3 : |x-y| < R \},
\end{equation}
\begin{equation}
	Q(z,R) = B(x,R) \times ]t-R^2,r[,
\end{equation}
\begin{equation}
	K(x',R) = \{ y' \in \R^2 : |x'-y'| < R \}.
\end{equation}
We denote $B(R) = B(0,R)$, $B = B(1)$, and similarly for $Q$ and $K$. Also, $B^+(R) = B(R) \cap \{ x_3 > 0 \}$, $B^+ = B^+(1)$, $Q^+(R) = B^+(R) \times ]-R^2,0[$, and $Q^+ = Q^+(1)$.

If $\Omega \subset \R^3$ is open and $I \subset \R$ is an interval, we define $Q_I = \Omega \times I$.

Let $1 \leq p,q \leq \infty$ and $m,n \in \N_0$. We will use the Lebesgue spaces $L_{p,q}(Q_I)$, where $p$ represents spatial integrability and $q$ time integrability,
as well as Sobolev spaces $W^{m,n}_{p,q}(Q_I)$,
where $m$ represents differentiability in space and $n$ differentiability in time. In the literature, notation of the type $L^q_t L^p_x(Q_I)$, $L^2_t H^1_x(Q_I)$, or $L_q(I;L_p(\Omega))$, $L_2(I;H^1(\Omega))$, etc., are also common, and we may occasionally use them. If a function space appears without a domain, e.g., $L_3$, then the domain is taken to be $\R^3$. We typically do not change our notation to reflect whether function spaces consist of scalar-, vector-, and matrix-valued functions.

Finally, we will not change notation when passing to subsequences.

\section{Truncation procedure}

\label{sec:truncation}





To begin, we give a definition of \emph{boundary suitable weak solution} (cf.~\cite{sereginshilkinsolonnikovboundarypartialreg2014,mikhailovshilkincurvedboundary2006,sereginshilkinsurvey}). Let $\Omega \subset \R^3$ be a (possibly unbounded) $C^2$ domain and $\Gamma \subset \p\Omega$ be relatively open. Let $I = ]S,T[$ be a finite open interval.
\begin{definition}[Boundary suitable weak solution]
\label{boundarysuitableweaksoldef}
We say that $(v,q)$ is a \emph{boundary suitable weak solution} of the Navier-Stokes equations in $Q_I = \Omega \times I$ vanishing on $\Gamma$ if
\begin{enumerate}
	\item for all bounded subdomains $\Omega' \subset \Omega$ with $\overline{\Omega'} \subset \Omega \subset \Gamma$ and all $S < S' < T$,
\begin{equation}
	\label{eq:vqrequirement}
	v \in L_{2,\infty}\cap W^{1,0}_2 (\Omega' \times ]S',T[), \quad q \in L_{\frac{3}{2}}(\Omega' \times ]S',T[),
\end{equation}
 and $v(\cdot,t)\big|_{\Gamma} = 0$ in the sense of trace for almost every $t \in I$,
\item
$(v,q)$ solves the Navier-Stokes equations on $Q_I$ in the sense of distributions:
\begin{equation}
\left\lbrace
\begin{aligned}
	\p_t v - \Delta v + v \cdot \nabla v + \nabla q &= 0 & &\text{ in } Q_I\\
	\div v &= 0 & &\text{ in } Q_I,
	\end{aligned}
	\right.
\end{equation}
\item
and $(v,q)$ satisfies the local energy inequality:
\begin{equation}
	\int_{\Omega} \zeta |v(x,t)|^2 \, dx  + 2 \int_0^t \int_\Omega \zeta |\nabla v|^2 \, dx \, dt' \leq $$ $$ \leq \int_0^t \int_{\Omega} |v|^2 (\p_t + \Delta) \zeta + (|v|^2 + 2q) v \cdot \nabla \zeta \, dx \,dt'
\end{equation}
for all non-negative $\zeta \in C^\infty_0((\Omega \cup \Gamma) \times ]S,T])$ and almost every $t \in I$.\footnote{Since $v \in C_w([0,T];L^2(\Omega))$, the local energy inequality is actually satisfied for every $t \in I$.}
\end{enumerate}
\end{definition}

One may use the local boundary regularity for the Stokes system in~\cite{shilkinvialov2013} (cf.~\cite{sereginnotelocalboundaryreg2009} for flat boundaries) and~\eqref{eq:vqrequirement} to bootstrap and obtain that each boundary suitable weak solution satisfies
\begin{equation}
	\label{eq:vqadditionalreg}
	(v,q) \in W^{2,1}_{\frac{9}{8},\frac{3}{2}} \times W^{1,0}_{\frac{9}{8},\frac{3}{2}}(\Omega' \times ]S',T[)
\end{equation}
for all $\Omega'$ and $S'$ as above. Definition~\ref{boundarysuitableweaksoldef} differs slightly from previous definitions in two ways: i)~We only require~\eqref{eq:vqrequirement} in $\Omega' \times ]S',T[$ rather than in the whole $Q_I$, and ii) Equation~\ref{eq:vqadditionalreg} is obtained as a consequence rather than directly imposed. Our definition works well for $\Omega=\R^3_+$ and solutions with infinite energy, as we encounter in the proof of Theorem~\ref{thm:l3}.

We now present the localisation procedure.
\begin{pro}[Truncation procedure]
\label{pro:truncation}
Let $\Omega \subset \R^3$ be a bounded $C^2$ domain, $x_0 \in \Omega$, and $R_0 > 0$ such that $\overline{\Omega_{x_0,R_0}} \subset \Omega \cup \Gamma$. Let $v$ be a boundary suitable weak solution on $\Omega \times ]0,T[$ ($T > 0$).

There exist $0 \leq \delta_1 < T$, $\Phi \in C^\infty(\R^3)$ ($0 \leq \Phi \leq 1$), and vector fields $w$ and $f$ such that the following hold:
\begin{enumerate}[1.]
	\item The vector fields $w$ and $f$ satisfy
\begin{equation}
	\label{wespaces}
	w \in W^{2,1}_{p,\frac{3}{2}} \cap L_\infty \cap W^{1,0}_{2}(\Omega \times ]\delta_1,T[) \text{ for all } p \geq 1,
\end{equation}
\begin{equation}
	\label{fespaces}
	f \in L_{p,\frac{3}{2}}(\Omega \times ]\delta_1,T[) \text{ for all } p \geq 1.
\end{equation}
\item $(V,Q)$ defined by 
\begin{equation}
	(V,Q) := (\Phi v + w, \Phi q)
\end{equation}
solves the Navier-Stokes equations with forcing term $f$ in the sense of distributions on $\Omega \times ]\delta_1,T[$.
\item There exists $0 < \bar{R}_0 \leq R_0$ such that
\begin{equation}
	(V,Q) \equiv (v,q) \text{ on } \Omega_{x_0,\bar{R}_0} \times ]\delta_1,T[.
\end{equation}
\item Suppose $v \in L_\infty(\Omega \times ]\delta_1,S[)$ for some $S \in ]\delta_1,T]$. Then $V \in L_\infty(\Omega \times ]\delta_1,S[)$. Moreover, $V$ is the unique weak Leray-Hopf solution on $\Omega \times ]\delta_1,S[$ with initial data $V(\cdot,\delta_1)$ and forcing term $f$.
\item Suppose $\norm{v(\cdot,t)}_{L_3^\weak(\Omega_{x_0,R_0})} \leq M$ for some $t \in ]\delta_1,T[$. Then $\norm{V(\cdot,t)}_{L_3^\weak(\Omega)} \leq c_0M$, where $L_3^\weak$ is the Lorentz space $L^{3,\infty}$ and $c_0 > 0$ is an absolute constant.
\end{enumerate}
\end{pro}

 To control the $L_3^\weak$ norm of the truncated solutions in a uniform way, we need to control the constants in Bogovskii's operator. We summarize the necessary facts here:
\begin{lemma}[Bogovskii's operator]
\label{lem:bogovskii}
	Let $d \geq 2$ and $\Omega \subset \R^d$ be a bounded domain star-shaped with respect to a ball $B(x,R)$ compactly contained in $\Omega$.\footnote{That is, for each $y_1 \in B(x,R)$ and $y_2 \in \Omega$, the closed line segment connecting $y_1$ and $y_2$ lies within $\Omega$.} Suppose $A \geq {\rm diam}(\Omega)/R$.

	There exists a linear operator $B \: C^\infty_{0,\avg}(\Omega) \to C^\infty_0(\Omega)$
	satisfying (denote $w = Bg$) the equation
	\begin{equation}
	\label{eq:divprob}
	\left\lbrace
	\begin{aligned}
	\div w &= g & &\text{ in } \Omega \\
	w \big|_{\p \Omega} &= 0 & &\text{ on } \p\Omega.
	\end{aligned}
	\right.
	\end{equation}
	Here and in the sequel, $\avg$ denotes zero spatial average.

	Let $k \in \N_0$ and $1 < p < \infty$. Then, for all $g \in C^\infty_{0,\avg}(\Omega)$,
	\begin{equation}
	\label{timeindepbogest}
	\norm{\nabla w}_{W^{k}_{p}(\Omega)} \leq C(d,k,p,A) \norm{g}_{W^{k}_{p}(\Omega)}.
	\end{equation}
	with positive constant $C$ independent of $g$. Hence, $B$ extends uniquely to a bounded linear operator
	$B \: \mathring{W}^{k}_{p,\avg}(\Omega) \to \mathring{W}^{k+1}_{p}(\Omega)$
	solving~\eqref{eq:divprob}, where $\mathring{}$ denotes the closure of test functions.

	Let $I \subset \R$ be an open interval and $g \in L_1(I;L_{p,\avg}(\Omega))$.
	Consider the linear operator $B$ defined by applying the above operator at almost every time.
	If $\p_t g \in L_1(I;L_p(\Omega))$, then $B$ commutes with the time derivative:
	\begin{equation}
		\label{Bcommutewithpt}
		\p_t B(g) = B(\p_t g).
	\end{equation}
\end{lemma}
For the time-independent assertions, see~\cite{bogovskii} and \cite[Lemma III.3.1 \& Remark III.3.2]{galdi}. Similar results are true in the bounded Lipschitz setting. To prove~\eqref{Bcommutewithpt}, one may use the finite difference operator $D^h_t \varphi = \left( \varphi + \varphi(\cdot+h) \right)/h$ and $h \to 0^+$.

\begin{remark}[Bogovskii in $L_p^\weak$ in divergence form]
\label{rmk:bogweaklp}
Consider a bounded Lipschitz domain $\Omega$ which is also star-shaped with respect to a ball. Let $A>0$ be as in Lemma~\ref{lem:bogovskii}. 
For $1 < p < \infty$, suppose
\begin{equation}
	g \in W^{1,p}_{\div}(\Omega) := \left\lbrace g \in L_p(\Omega) : \div g \in L_p(\Omega) \right\rbrace
\end{equation}
with $g \cdot n|_{\p \Omega} = 0$. Hence, $\int_{\Omega} \div g \, dx = 0$. In this case, we have the divergence-form estimate (cf. Proposition 2.1 in~\cite{Hieber2016} or Theorem III.3.4 in~\cite{galdi})\footnote{Technically, the dependence on $A$ is not explicitly stated therein (their statements are for bounded Lipschitz domains), but it follows from the proof in~\cite{galdi}.}
\begin{equation}
	\label{eq:divformest}
	\norm{B(\div g)}_{L_p(\Omega)} \leq C(p,A) \norm{g}_{L_p(\Omega)}.
\end{equation}
Combining~\eqref{eq:divformest} with real interpolation, we may estimate the Bogovskii operator in Lorentz spaces:
\begin{equation}
	\label{eq:bogweaklp}
	\norm{B(\div g)}_{L_p^\weak(\Omega)} \leq C(p,A) \norm{g}_{L_p^\weak(\Omega)}.
\end{equation}
\end{remark}

 Cones are convenient for constructing star-shaped domains, so we let
\begin{equation}
	E(r) = \{ |x_3-r|^2 > |x'|^2 \text{ and }  x_3 < r \}
\end{equation}
denote the cone of angle $\pi/2$ and vertex $r e_3$ pointing in the $e_3$ direction.

The following lemma is elementary, and we omit the proof:

\begin{lemma}[A star-shaped domain]
\label{lem:starlikelemma}
There exists $0 < N^* \ll 1$ such that, if $\varphi \in C^2(\overline{K(2)})$ satisfies
\begin{equation}
	\label{phisatisfiesnstar}
	\varphi(0) = 0, \; (\nabla \varphi)(0) = 0, \text{ and } \norm{\varphi}_{C^2(K(2))} \leq N^*,
\end{equation}
then, for all $0 < r \leq 1/4$, the bounded domain
\begin{equation}
	\label{eq:odef}
	\cO(\varphi,r) := \left\lbrace ]-1,1[^3 \setminus E(r) \right\rbrace \cap \left\lbrace x_3 > \varphi(x') \right\rbrace
\end{equation}
is star-shaped with respect to the ball $B^* = B(3e_3/4 ,1/16)$ and Lipschitz.
\end{lemma}

With this in mind, we prove Proposition~\ref{pro:truncation}.

\begin{proof}[Proof of Proposition~\ref{pro:truncation}]
We present only the case $x_0 \in \p\Omega$.\footnote{For the case $x_0 \in \Omega$, one truncates in a small annulus $B(2r_1) \setminus B(r_1)$ about $x_0=0$ and the version of Bogovskii's operator for Lipschitz domains rather than star-shaped domains, see \cite[Chapter III]{galdi}.} Because $\Omega$ is a bounded $C^2$ domain, we may use the symmetries of the Navier-Stokes equations to obtain the following situation: $x_0 = 0$, $R_0 \geq 2$, and $\Omega \cap B(2) = \{ |x| < 2 : x_3 > \varphi(x') \}$
for a function $\varphi \in C^2(\overline{K(2)})$ satisfying~\eqref{phisatisfiesnstar}.

We require a lemma which essentially follows from boundary partial regularity, see~\cite{neustupapenel1999,kukavicarusinzianeJMFM2017} and related works for similar results.
\begin{lemma}[Regular annulus lemma]
\label{lem:regannlem}
There exist $0 < r_1 < r_2 \leq 1/4$ and $0 < \delta_1 < \delta_2 \leq T$ such that, on the following sets, $v$ is essentially bounded and $(v,q)$ belongs to $W^{2,1}_{p,\frac{3}{2}} \times W^{1,0}_{p,\frac{3}{2}}$ for all $p \geq 1$:
\begin{equation}
	\label{set1}
	\left( \left\lbrace E(r_2) \setminus E(r_1) \right\rbrace \cap \left\lbrace |x| < 2 : x_3 > \varphi(x') \right\rbrace \right) \times ]\delta_1,T[,
\end{equation}
\begin{equation}
	\label{set2}
	\left\lbrace |x| < 2 : x_3 > \varphi(x') \right\rbrace \times ]\delta_1,\delta_2[.
\end{equation}
\end{lemma}
\begin{proof}
The $L_\infty$ assertion follows from the boundary partial regularity proven in~\cite{sereginshilkinsolonnikovboundarypartialreg2014}. Indeed, suppose that for each $0 < \delta_1 < \delta_2 \leq T$, $v$ were not essentially bounded on~\eqref{set2}. Then $v$ would necessarily have a singular point in $\overline{\Omega}$ at every time $t \in ]0,T]$, contradicting that the one-dimensional parabolic Hausdorff measure of the singular set is zero. We obtain $r_1$ and $r_2$ by similar reasoning. The higher regularity assertion follows from a bootstrapping argument using the local boundary regularity theory for the non-stationary Stokes equations proven in~\cite{shilkinvialov2013} as long as one slightly increases $\delta_1$ and $r_1$ and slightly decreases $r_2$.
\end{proof}

We now continue with the proof of Proposition~\ref{pro:truncation}. For convenience, we denote $I = ]\delta_1,T[$ and $\cO = \cO(\varphi,r_1)$, which was defined in~\eqref{eq:odef}. We will justify the assertions of Proposition~\ref{pro:truncation} in order.

Let $\Phi \in C^\infty(\R^3)$ ($0 \leq \Phi \leq 1$) with $\Phi \equiv 1$ in a neighborhood of $E(r_1)$ and $\Phi \equiv 0$ in a neighborhood of $\R^3 \setminus E(r_2)$. We introduce a correction $w$ solving
\begin{equation}
	\label{eq:divergenceeqn}
	\left\lbrace
	\begin{aligned}
	\div w &= -\nabla \Phi \cdot v & &\text{ in } \cO \times I \\
	w &= 0 & &\text{ on } \p\cO \times I.
	\end{aligned}
	\right. 
	\end{equation}
	According to Lemma~\ref{lem:starlikelemma}, $\cO$ is star-shaped with respect to the ball $B^*$. Hence, we may apply Bogovskii's operator, whose properties are recalled in Lemma~\ref{lem:bogovskii} (with $A=64$), to solve~\eqref{eq:divergenceeqn}.

We claim
\begin{equation}
\label{vanishingtracephiv}
	\nabla \Phi \cdot v \in W^{1,1}_{p,\frac{3}{2}}(\cO \times I) \text{ and } 
	\nabla \Phi \cdot v \in L_{\frac{3}{2}}(I;\mathring{W}^{1}_{p,\avg}(\cO)) \text{ for all } p \geq 1.
\end{equation}
Recall that ${\rm supp}(\nabla\Phi) \subset \overline{E(r_2)} \setminus E(r_1)$. With this in mind, the regularity in~\eqref{vanishingtracephiv} is known from Lemma~\ref{lem:regannlem}, the vanishing trace follows from the no-slip boundary condition, and the zero average is verified by
\begin{equation}
	\int_{\cO} \nabla \Phi \cdot v \, dx = \int_{\Omega} \nabla \Phi \cdot v \, dx = \int_\Omega \div(\Phi v) \, dx = \int_{\p \Omega} \Phi v \cdot n \, dS = 0,
\end{equation}
since $\Phi v(\cdot,t)|_{\p \Omega} = 0$.
Thus, according to Lemma~\ref{lem:bogovskii},
\begin{equation}
	w, \nabla w, \p_t w \in L_{\frac{3}{2}}(I;\mathring{W}^{1}_{p}(\cO)) \text{ for all } p \geq 1.
\end{equation}
Moreover, we may extend $w$ by zero to obtain $w \in W^{2,1}_{p,\frac{3}{2}}(\Omega \times I)$ for all $p \geq 1$. The proof of~\eqref{wespaces} is concluded by using parabolic Sobolev embedding.

Next, we define $(V,Q) := (\Phi v + w, \Phi q)$. A direct computation shows that $(V,Q)$ satisfies the Navier-Stokes equations in the sense of distributions on $\Omega \times I$ with forcing term
\begin{equation}
\begin{gathered}
	f := (\p_t - \Delta) \Phi v - 2 \nabla \Phi \cdot \nabla v + \Phi v \cdot (v \otimes \nabla \Phi) + (\Phi^2 - \Phi) v \cdot \nabla v \\ + (\p_t - \Delta) w + \Phi v \cdot \nabla w + w \cdot \nabla (\Phi v) + w \cdot \nabla w + \nabla \Phi q.
	\end{gathered}
	\end{equation}
Then~\eqref{fespaces} follows from the known properties of $v$, $q$, $w$, and $\Phi$. In particular, we exploit that $\Phi v$ is essentially bounded on the support of $w$.

Notice that $\Phi \equiv 1$ and $w \equiv 0$ on $E(r_1) \cap \{ \varphi(x') > x_3 \}$. Together with~\eqref{phisatisfiesnstar}, this implies that there exists $0 < \bar{R}_0 < 2$ satisfying $(V,Q) \equiv (v,q)$ on $B(\bar{R}_0) \cap \{ \varphi(x') > x_3 \}$.

Let us assume that $v \in L_{\infty}(\Omega \times ]\delta_1,S[)$ for some $S \in I$. Because $v$ and $w$ belong to the energy space $L_{2,\infty} \cap W^{1,0}_2(\Omega \times I)$, $V$ belongs to the energy space as well. In addition, $V(\cdot,t)|_{\p \Omega} = 0$ for a.e. $t \in I$. Next, by our assumption, $V$ and $f$ have enough integrability to prove the energy equality on $\Omega \times ]\delta_1,S[$ directly (see Theorem~1.4.1, p. 272, in~\cite{Sohrbook}, for example). Together, these facts imply that $V$ is a weak Leray-Hopf solution on $\Omega \times ]\delta_1,S[$ with initial data $V(\cdot,\delta_1)$ and forcing term $f$. Our assumptions are enough to prove weak-strong uniqueness in the standard way (see Theorem~1.5.1, p. 276, in~\cite{Sohrbook}).
 
 Finally, assume that $\norm{v(\cdot,t)}_{L_3^\weak(B(2))} \leq M$ for some $t \in I$. Then $\norm{w(\cdot,t)}_{L_3^\weak(\cO)} \leq C(A)M$ according to~\eqref{eq:bogweaklp} in Remark~\ref{rmk:bogweaklp}. We use $A=64$ to complete the proof.
\end{proof}

\section{Local concentration of $L_3$ norm}
\label{sec:l3}

In this section, we state and prove a more quantitative version of Theorem~\ref{cor:l3} in $L_3^\weak$ by following the scheme explained in the introduction.

\begin{theorem}[Behavior of $L_3^\weak$ norm]
\label{thm:l3}
Let $\Omega \subset \R^3$ be a bounded $C^2$ domain and $\Gamma \subset \p\Omega$ be relatively open. Let $x^* \in \Omega\cup \Gamma$ and $R>0$ such that $\overline{\Omega_{x^*,R}} \subset \Omega\cup \Gamma$. 

For each $M > 0$, there exists a constant $\epsilon = \epsilon(\Omega,M) > 0$ such that the following property holds:

Let $v$ be a boundary suitable weak solution in $\Omega \times ]0,1[$ vanishing on $\Gamma$ and satisfying
\begin{equation}
\label{eq:vboundeforabit}
 v\in L_{\infty}(\Omega \times ]0,t[) \; \text{ for all } 0<t<1.
 \end{equation}
If there exists a sequence $t_{k}\uparrow 1$ such that
\begin{equation}\label{eq:localL3M}
\sup_{k \in \N} \norm{v(\cdot, t_k)}_{L_3^\weak(\Omega_{x^*,R})} \leq M
\end{equation}
and
\begin{equation}
\label{eq:disttoL}
	\dist_{L_3^\weak} \left( v(\cdot+x^*,1), \mathbb{L} \right) \leq \epsilon,
\end{equation}
then
\begin{equation}
\label{eq:conczstartisregular}
	z^* = (x^*,1) \text{ is a regular point of } v.
\end{equation}
\end{theorem}

Here, $\mathbb{L}$ is the space of functions $f \in L_3^\weak$ satisfying
\begin{equation}
	\norm{f}_{L_3^\weak(B(r))} \to 0 \text{ as } r \to 0^+.
\end{equation}
Equivalently, $\mathbb{L}$ is the $L_3^\weak$ closure of the set of functions $g \in L_3^\weak$ that are smooth in a neighborhood of the origin.

\begin{remark}
	It seems possible to remove the dependence of $\epsilon$ on $\Omega$. For example, one can use a notion of `weak $L^{3,\infty}$ solution' in $\Omega \times ]0,1[$ (analogous to the notion in~\cite{barkersereginsverakstability} in the whole space), rather than the Calder{\'o}n-type splitting, to obtain energy estimates.
	In the boundary case, the limit solution $v^\infty$ would be a `weak $L^{3,\infty}$ solution' in $\R^3_+ \times ]0,1[$. A Liouville theorem analogous to Lemma~\ref{lem:Liouvilletheorem} and \cite[Remark 4.2]{globalweakbesov} should be possible for such solutions and complete the proof.
	It may also be possible to prove a version with control in $\dot B^{-1+\frac{3}{p}}_{p,\infty}$ along a sequence of times; the Calder{\'o}n-type splittings in~\cite{barkerweakstrong,albrittonblowupcriteria} and existence theory in~\cite{gigabesovspaces} may be useful here. Lastly,~\eqref{eq:disttoL} can probably be weakened using Besov spaces.
\end{remark}

\begin{proof}[Proof of Theorem~\ref{thm:l3}] The proof is by contradiction. Let $v$ be a boundary suitable weak solution in $\Omega \times ]0,1[$ vanishing on~$\Gamma$. Assume that \eqref{eq:vboundeforabit}-\eqref{eq:disttoL} are satisfied, where $\epsilon > 0$ is to be determined. For contradiction, assume that~\eqref{eq:conczstartisregular} is not satisfied. That is,
\begin{equation}
	z^* = (x^*,1) \text{ is a singular point of } v.
\end{equation}
By translating our domain, we may assume that $x^* = 0$. 

 \subsection{Truncation and rescaling}

 \subsubsection*{Step 1: Apply the trunction procedure}

 To begin, we apply the truncation procedure in Proposition~\ref{pro:truncation}. By slightly zooming in, we may set $\delta_1 = 0$. We summarize the resulting situation below:
\begin{equation}
	V \in L_\infty(\Omega \times ]0,t[) \text{ for all } t \in ]0,1[
\end{equation}
is the unique weak Leray-Hopf solution in $\Omega \times ]0,1[$ with initial data $V(\cdot,0)$ and forcing term
\begin{equation}\label{forcingspcaes}
  f \in L_{p,\frac{3}{2}}(\Omega \times ]0,1[) \text{ for all } p \geq 1.
 \end{equation}
 Its associated pressure is denoted by $Q$, and there exists $\bar{R} > 0$ such that
\begin{equation}
	\label{VQequivvqbarR}
	(V,Q) \equiv (v,q) \text{ on } \Omega_{x^*,\bar{R}} \times ]0,1[.
\end{equation}
 Additionally, \eqref{eq:localL3M} implies
 \begin{equation}\label{VlocalL3}
 \sup_{k \in \N} \norm{V(\cdot,t_k)}_{L_3^\weak(\Omega)} \leq M'<\infty,
 \end{equation}
 where $M'$ depends only on $M$.

 \subsubsection*{Step 2: Rescaling and key norm relations}

If $x^* \in \p\Omega$, then we rotate the original coordinate system such that, in the new coordinates,
\begin{equation}
	\label{eq:part1BR0}
	\Omega \cap B(R_0) = \left\lbrace x = (x',x_3) \in B(R_0): x_{3}> \varphi(x') \right\rbrace,
\end{equation}
where $R_0$ and $N_0$ are positive constants and $\varphi \in C^2(\overline{K(R_0)})$ satisfies
\begin{equation}
	\label{eq:part1varphi}
	\varphi(0)=0, \; \nabla \varphi(0)=0, \text{ and } [\varphi]_{C^2(K(R_0))}\leq N_0.
\end{equation}
Furthermore, we take $R_0 \leq \bar{R}$ in~\eqref{VQequivvqbarR}.

 Throughout, we denote $R_{k}:= \sqrt{1-t_{k}}$. 
 We rescale
 \begin{equation}\label{Vrescale}
 V^{(k)}(y,s):= R_{k} V(R_{k} y, t_{k}+ R_k^{2} s)
 \end{equation}
 and
 \begin{equation}\label{frescale}
 f^{(k)}(y,s):= R_k^3 f(R_{k} y, t_{k}+ R_k^{2} s).
 \end{equation}
 The above functions are defined on $\Omega_{k} \times ]0,1[$.
 Here, $\Omega_{k}:=\Omega/R_k$.
 From (\ref{VlocalL3}), we see that
 \begin{equation}\label{rescaleVL3}
 \sup_{k \in \N} \|V^{k}(\cdot,0)\|_{L_3^\weak(\Omega_{k})}=M'<\infty.
 \end{equation}
 Furthermore,
 \begin{equation}\label{rescalefnorm}
 \norm{f^{(k)}}_{L_{2,1}(\Omega_{k} \times ]0,1[)} \leq \norm{f^{(k)}}_{L_{2,\frac{3}{2}}(\Omega_{k} \times ]0,1[)} = R_k^{\frac{1}{6}} \cF,
 \end{equation}
 where
 \begin{equation}
	\cF = \norm{f}_{L_{2,\frac{3}{2}}(\Omega \times ]0,1[)}.
 \end{equation}
 We denote $u_0^{(k)} = V^{(k)}(\cdot,0)$.

 \subsection{\emph{A priori} estimates for rescaled truncated solution}


 In the sequel, $C$ may implicitly depend on $\Omega$ and $M'$.

 \subsubsection*{Step 3: Energy estimates}
 
We now prove \emph{a priori} energy estimates for $V^{(k)}$ using a Calder{\'o}n-type splitting~\cite{calderon90} of the initial data. Similar methods were exploited in previous papers~\cite{Jiasver2013,barkerweakstrong,albrittonblowupcriteria}, for example.

We decompose
\begin{equation}
	u_0^{(k)} = \tilde{u_0}^{(k)} + \bar{u_0}^{(k)},
\end{equation}
\begin{equation}
	\tilde{u_0}^{(k)} = \bP \left( \mathbf{1}_{ \{ |u_0^{(k)}| < 1/2 \} } u_0^{(k)} \right),
\end{equation}
\begin{equation}
	\label{eq:tildebaru0bound}
	\norm{\tilde{u_0}^{(k)}}_{L_2(\Omega_k)} + \norm{\bar{u_0}^{(k)}}_{L_{\frac{10}{3}}(\Omega_k)}  \leq C.
\end{equation}
Let
 \begin{equation}
	V^{(k)}:= U^{(k)}+L^{(k)}. 
 \end{equation}
 Here, $L^{(k)}:= L^{(k)}(\bar{u_0}^{(k)})$ satisfies
 \begin{equation}\label{linearequation1}
 \left\lbrace
 \begin{aligned}
 \p_t L^{(k)}-\Delta L^{(k)} + \nabla \pi^{(k)} &= 0 & &\text{ in } \Omega_{k} \times \R_+ \\
 \div L^{(k)}&=0 & &\text{ in } \Omega_{k} \times \R_+ \\
 L^{(k)}\big|_{\partial \Omega_{k}}&=0 & &\text{ in } \p\Omega_k \times \R_+ \\
 L^{(k)}(\cdot,0) &= \bar{u_0}^{(k)} & &\text{ in } \Omega_k, \\
 \end{aligned}
 \right.
 \end{equation}
 whereas $U^{(k)}$ satisfies the following perturbed Navier-Stokes system:
 \begin{equation}
 \left\lbrace
 \begin{aligned}
 \p_t U^{(k)}-\Delta U^{(k)} + V^{(k)} \cdot \nabla U^{(k)} + U^{(k)} \cdot \nabla L^{(k)} +\nabla P^{(k)}&= f^{(k)} - L^{(k)} \cdot \nabla L^{(k)} & &\text{ in } \Omega_{k} \times ]0,1[ \\
 \div L^{(k)}&=0 & &\text{ in } \Omega_{k} \times ]0,1[ \\
 U^{(k)}\big|_{\partial \Omega_{k}}&=0 & &\text{ in } \p\Omega_k \times ]0,1[ \\
 U^{(k)}(\cdot,0) &= \tilde{u_0}^{(k)} & &\text{ in } \Omega_k. \\
 \end{aligned}
 \right.
 \end{equation}
 By the smoothing estimates for the Stokes semigroup in \cite[Theorem 5.1]{Sol1977} and the estimate~\eqref{eq:tildebaru0bound} for $\bar{u_0}^{(k)}$, we have the following for all $0<t<1$ and $\frac{10}{3} \leq p \leq \infty$:
 \begin{equation}\label{linearestzeroforce}
 \norm{L^{(k)}(t)}_{L_{p}(\Omega_{k})}\leq \frac{C(\Omega)M^{'}}{t^{\frac{3}{2}(\frac{3}{10}-\frac{1}{p})}},
 \end{equation}
 \begin{equation}\label{gradientlinearestzeroforce}
 \norm{\nabla L^{(k)}(t)}_{L_{p}(\Omega_{k})}\leq \frac{C(\Omega) M^{'}}{t^{\frac{3}{2}(\frac{3}{10}-\frac{1}{p})+\frac{1}{2}}}.
 \end{equation}
 Notice that $L^{(k)}$ belongs to the energy space (see Lemma 1.5.1, p. 204, of \cite{Sohrbook}, for example); hence, $U^{(k)}$ does as well. Since $L^{(k)} \in L_5(\Omega^{(k)} \times ]0,1[)$ due to~\eqref{linearestzeroforce}, we may infer that $U^{(k)}$ satisfies the energy equality for $0 < t < 1$ (see Theorem 2.3.1, p. 226, of \cite{Sohrbook}, for example).
 %
 %
 %
 Namely,
 \begin{equation}\label{Urescaledenergyequality}
\frac{1}{2}\norm{U^{(k)}(\cdot,t)}_{L_{2}(\Omega_{k})}^2+\int_{0}^{t}\int_{\Omega_{k}} |\nabla U^{(k)}|^2  dy \, ds=\frac{1}{2} \norm{\tilde{u_0}^{(k)}}_{L_2(\Omega_k)}^2 + $$
$$+ \int_{0}^{t}\int_{\Omega_{k}} L^{(k)} \otimes L^{(k)} :\nabla U^{(k)} + f^{(k)}\cdot U^{(k)} +U^{(k)}\otimes L^{(k)}: \nabla U^{(k)} dy \, ds.
 \end{equation}
 Using H\"{o}lder's inequality and the estimate~\eqref{eq:tildebaru0bound} for $\tilde{u_0}^{(k)}$, it can be shown that 
 \begin{equation}\label{Urescaledenergyest}
 \norm{U^{(k)}(\cdot,t)}_{L_{2}(\Omega_{k})}^2+\int_{0}^{t}\int_{\Omega_{k}} |\nabla U^{(k)}|^2 dyds\leq C + $$ 
 $$  + C \int_{0}^{t}\int_{\Omega_{k}} |L^{(k)}|^4 dy+ 
\norm{L^{(k)}(\cdot,s)}_{L_\infty(\Omega_k)}^2 \norm{U^{(k)}(\cdot,s)}_{L_{2}(\Omega_{k})}^2
 +\norm{f^{(k)}(\cdot,s)}_{L_{2}(\Omega_{k})}\norm{U^{(k)}(\cdot,s)}_{L_{2}(\Omega_{k})} \, ds .
 \end{equation}
 This, the smoothing estimate~\eqref{linearestzeroforce} (with $p=4,\infty$) and~\eqref{rescalefnorm} concerning $f^{(k)}$ imply
 \begin{equation}\label{Urescaledenergyest1}
 \norm{U^{(k)}}_{L_{\infty}(0,t; L_{2}(\Omega_{k})}^2+ \int_{0}^{t}\int_{\Omega_{k}} |\nabla U^{(k)}|^2 dyds\leq C +$$
 $$ +   C t^{\frac{7}{10}}+CR_{k}^{\frac{1}{6}} \cF \norm{U^{(k)}}_{L_{2,\infty}(\Omega_{k} \times ]0,t[)}+ C \int_{0}^{t} \frac{1}{s^{\frac{9}{10}}} \norm{U^{(k)}(\cdot,s)}_{L_{2}(\Omega_{k})}^2 \, ds.
 \end{equation}
 Applying Young's inequality and using that $t<1$, we obtain
 \begin{equation}\label{Urescaleenergyest2}
 \norm{U^{(k)}(\cdot,t)}_{L_{2}(\Omega_{k})}^2+ \int_{0}^{t}\int_{\Omega_{k}} |\nabla U^{(k)}|^2 dy\, ds \leq  C+CR_{k}^{\frac{1}{3}} \cF^2+ C \int_{0}^{t} \frac{1}{s^{\frac{9}{10}}} \norm{U^{(k)}(\cdot,s)}_{L_{2}(\Omega_{k})}^2 ds.
 \end{equation}
 An application of the generalized Gronwall lemma then gives that for all $0<t<1$:
 \begin{equation}\label{Urescaleuniformenergyest}
 \norm{U^{(k)}(\cdot,t)}_{L_{2}(\Omega_{k})}^2+ \int_{0}^{t}\int_{\Omega_{k}} |\nabla U^{(k)}|^2 dy \, ds \leq C \times \left(1+ R_{k}^{\frac{1}{3}} \cF^2 \right).
 \end{equation}
 Using this, interpolation of Lesbesgue spaces, the Sobolev embedding theorem, and the smoothing estimate~\eqref{linearestzeroforce} (with $p=10/3$) gives
 \begin{equation}\label{Vrescaled10/3uniform}
 \norm{V^{(k)}}_{L_{\frac{10}{3}}(\Omega_{k} \times ]0,1[)}\leq C \times \left(1 + R_{k}^{\frac{1}{6}} \cF \right).
 \end{equation}
 
 \subsubsection*{Step 4: Maximal regularity estimates}

 Next, we use maximal regularity estimates to obtain estimates on the time derivative and the pressure. We remark that, in the whole-space setting, one can simply represent the pressure in terms of $u$ using Riesz transforms and then estimate the time derivative (in a negative Sobolev space, say) from the equation. This method is not available to us here.

 We decompose $-V^{(k)}\cdot\nabla V^{(k)}+f^{(k)}$ as
 \begin{equation}
	-V^{(k)}\cdot\nabla V^{(k)}+f^{(k)}=f^{(k),0}+f^{(k),1}+f^{(k),2}+f^{(k),3}.
 \end{equation}
Here,
\begin{equation}\label{fk0def}
f^{(k),0}:= f^{(k)},
\end{equation}
\begin{equation}\label{fk1def}
f^{(k),1}:= -U^{(k)}\cdot \nabla U^{(k)}
\end{equation}
\begin{equation}\label{fk2def}
f^{(k),2}:=-L^{(k)}\cdot\nabla L^{(k)}-U^{(k)}\cdot\nabla L^{(k)}
\end{equation}
\begin{equation}\label{fk3def}
f^{(k), 3}:= -L^{(k)}\cdot \nabla U^{(k)}.
\end{equation}
From (\ref{rescalefnorm}), we have
\begin{equation}\label{fk0est}
\norm{f^{(k),0}}_{L_{2,\frac{3}{2}}(\Omega_{k} \times ]0,1[)}\leq R_{k}^{\frac{1}{6}} \cF
\end{equation}
From (\ref{Urescaleuniformenergyest}), Sobolev embedding, interpolation, and H{\"o}lder's inequality, we have
\begin{equation}\label{fk1est}
\norm{f^{(k),1}}_{L_{\frac{9}{8},\frac{3}{2}}(\Omega_{k} \times ]0,1[)}\leq C \times \left(1+ R_{k}^{\frac{1}{3}} \cF^2 \right).
\end{equation} 
Using (\ref{linearestzeroforce})-(\ref{gradientlinearestzeroforce}) and (\ref{Urescaleuniformenergyest}), we infer that
\begin{equation}\label{fk2+3est}
\norm{f^{(k),2}}_{L_{\frac{10}{3}}(\Omega_{k} \times ]\frac{1}{4}, 1[)}+\norm{f^{(k),3}}_{L_{2}(\Omega_{k} \times ]\frac{1}{4}, 1[)}\leq C \times \left(1+ R_{k}^{\frac{1}{6}} \cF \right).
\end{equation}

In order to apply maximal regularity, it is convenient to get rid of the initial condition. Let us fix a smooth cut-off function $\chi$ such that
\begin{equation}
\chi(t):= \begin{cases}
1 & \text{ if } 1/4<t<2 \\
 0 & \text{ if } 0<t<1/8.
\end{cases}
\end{equation}
Using the uniqueness and maximal regularity results for the linear Stokes system in \cite[Theorem 2.8]{Gigasohr91}, we may split $\chi V^{(k)}$ and $\chi Q^{(k)}$ 
in the following way:
\begin{equation}
\label{chiV}
	\chi V^{(k)}=V^{(k),0}+V^{(k),1}+V^{(k),2}+V^{(k),3},
\end{equation}
\begin{equation}
\label{chiQ}
	\chi Q^{(k)}= Q^{(k),0}+Q^{(k),1}+Q^{(k),2}+Q^{(k),3},
\end{equation}
for $(x,t)\in \Omega_{k} \times ]0,1[$. Here,
\begin{equation}
\left\lbrace
\begin{aligned}
\p_t V^{(k),i}-\Delta V^{(k),i}+\nabla Q^{(k),i} &= g^{(k),i} & &\text{ in } \Omega_k \times ]0,1[ \\
\div V^{(k),i}&=0 & &\text{ in } \Omega_k \times ]0,1[\\
V^{(k),i}\big|_{\p \Omega_{k}}&=0 & &\text{ in } \p\Omega_k \times ]0,1[\\
V^{(k),i}(\cdot,0)&=0 & &\text{ in } \Omega_k
\end{aligned}
\right.
\end{equation}
for $i=0,\hdots,3$, where
\begin{equation}
	g^{(k),i} := \chi(t) f^{(k),i} - \delta_{i2}\chi'(t) V^{(k)},
\end{equation}
and $\delta_{i2}$ is the Kronecker delta.

 Using (\ref{Urescaleuniformenergyest}),  (\ref{fk0est})-(\ref{fk2+3est}) and the maximal regularity estimates in \cite[Theorem 2.8]{Gigasohr91}, it  follows that
 \begin{equation}\label{Vk0est}
 \|\partial_tV^{(k),0}\|_{L_{ 2, \frac{3}{2}}(\Omega_{k} \times ]0,1[)}+\|\nabla^2 V^{(k),0}\|_{L_{ 2, \frac{3}{2}}(\Omega_{k} \times ]0,1[)}+\|\nabla Q^{(k),0}\|_{L_{2, \frac{3}{2}}(\Omega_{k} \times ]0,1[)} $$$$\leq c\|g^{(k),0}\|_{L_{ 2, \frac{3}{2}}(\Omega_{k} \times ]0,1[)} \leq C_{0} \times R_{k}^{\frac{1}{6}} \cF,
 \end{equation}
\begin{equation}\label{Vk1est}
\|\partial_tV^{(k),1}\|_{L_{\frac 9 8, \frac 3 2}(\Omega_{k} \times ]0,1[)}+\|\nabla^2 V^{(k),1}\|_{L_{\frac 9 8, \frac 3 2}(\Omega_{k} \times ]0,1[)}+\|\nabla Q^{(k),1}\|_{L_{\frac 9 8, \frac 3 2}(\Omega_{k} \times ]0,1[)} $$$$\leq c\|g^{(k),1}\|_{L_{\frac 9 8, \frac 3 2}(\Omega_{k} \times ]0,1[)} \leq C_{1} \times \left(1+ R_{k}^{\frac{1}{3}} \cF^2 \right), 
\end{equation}
\begin{equation}\label{Vk2est}
\|\partial_tV^{(k),2}\|_{L_{\frac{10}{3}}(\Omega_{k} \times ]0,1[)}+\|\nabla^2 V^{(k),2}\|_{L_{\frac{10}{3}}(\Omega_{k} \times ]0,1[)}+\|\nabla Q^{(k),2}\|_{L_{\frac{10}{3}}(\Omega_{k} \times ]0,1[)}\leq $$$$\leq c\|g^{(k),2}\|_{L_{\frac{10}{3}}(\Omega_{k} \times ]0,1[)} \leq C_{2} \times \left(1+ R_{k}^{\frac{1}{6}} \cF \right)
\end{equation}
and
\begin{equation}\label{Vk3est}
\|\partial_tV^{(k),3}\|_{L_{2}(\Omega_{k} \times ]0,1[)}+\|\nabla^2 V^{(k),3}\|_{L_{2}(\Omega_{k} \times ]0,1[)}+\|\nabla Q^{(k),3}\|_{L_{2}(\Omega_{k} \times ]0,1[)}\leq $$$$\leq c\|g^{(k),3}\|_{L_{2}(\Omega_{k} \times ]0,1[)} \leq C_{3} \times \left(1+ R_{k}^{\frac{1}{6}} \cF \right).
\end{equation}
Furthermore, from (\ref{linearestzeroforce})-(\ref{gradientlinearestzeroforce}) and (\ref{Urescaleuniformenergyest}), we have that for any finite $a>0$:
\begin{equation}\label{Vkuniformenergyest}
\norm{V^{(k)}}_{L_{2,\infty}\left(\Omega_{k,x^*,a} \times ]\frac{1}{4},1[\right)} +\norm{\nabla V^{(k)}}_{L_{2}\left(\Omega_{k,x^*,a} \times ]\frac{1}{4},1[\right)} \leq C_{4}(a) \times \left(1+ R_{k}^{\frac{1}{6}} \cF \right),
\end{equation}
where $\Omega_{k,x^*,a} = \Omega_{k}\cap B(x^*/R_k,a)$.

\subsection{Conclusion}
In the sequel, we suppose $x^* \in \Gamma$. The interior case is described in Remark~\ref{rmk:interiorcase}.

\subsubsection*{Step 5: Flattening the boundary}

Recall $R_0$, $N_0$ and $\varphi$ satisfying~\eqref{eq:part1BR0}-\eqref{eq:part1varphi} in Step 2.

Considering the rescaled domains $\Omega_k = \Omega/R_k$, we have
\begin{equation}\label{domainrescalechange}
\Omega_{k} \cap B(R_0/R_{k}) = \left\lbrace |x| \leq R_0/R_k : x_{3}> \varphi_k(x') \right\rbrace.
\end{equation}
where $\varphi_k(x') := \varphi ( R_k x')/R_{k}$.
Obviously,
\begin{equation}\label{almostflat}
[\varphi_k]_{C^{2}(K(R_0/R_k))} \leq 3R_{k} N_0.
\end{equation}
We make the change of coordinates\footnote{Note the difference between $\phi$ and $\varphi$.} 
\begin{equation}\label{flatteningtheboundary}
x= \phi_k(y):=(y_1, y_2, y_3- \varphi_k(y_1,y_2)),
\end{equation}
\begin{equation}
	(\phi_k)^{-1}(x):= (x_1, x_2, x_3+\varphi_k(x_1,x_2) ).
\end{equation}
Using (\ref{almostflat}), we see that for $k$ sufficiently large, we have that for any $R\leq R_0/(2R_{k})$:
\begin{equation}\label{coordinatemappings1}
B^{+}(R) \subset \phi_k\left(\Omega_{k} \cap B(3R/2) \right) \subset B^{+}(2R),
\end{equation}
\begin{equation}\label{coordinatrmapping2}
(\phi_k)^{-1}(B^{+}(R)) \subset \left(\Omega_{k} \cap B(3R/2)\right) \subset (\phi_k)^{-1}(B^{+}(2R)).
\end{equation}
Fix $R>0$ and consider $k\geq \bar{k}(R,\Omega)$ sufficiently large, such that $R \leq R_0/(2R_k)$ and $[\varphi]_{C^2(K(R))} \leq \mu^*/R$, where $\mu^*$ is defined in Appendix~\ref{sec:persistenceofsing}. We define
\begin{equation}\label{vkchangeofcoordiantes}
\hat{v}^{(k)} := v^{(k)} \circ (\phi_k)^{-1},
\end{equation}
\begin{equation}\label{qkchangeofcordinates}
\hat{q}^{(k)} := q^{(k)} \circ (\phi_k)^{-1}.
\end{equation}
Then $(\hat{v}^{(k)}, \hat{q}^{(k)}, \varphi_k)$ is a boundary suitable weak solution of the flattened Navier-Stokes equations in $B^{+}(R) \times ]0,1[$.
Furthermore,
\begin{equation}\label{singpointflattened}
(0,1) \text{ is a singular point of } \hat{v}^{(k)}, \text{ for all } k\in\mathbb{N}. 
\end{equation}

\subsubsection*{Step 6. Passage to the limit}

By increasing $\bar{k}(R,\Omega)$ (and considering $k \geq \bar{k}$), we may ensure that $R_k^{\frac{1}{6}} \cF \leq 1$ in~\eqref{rescalefnorm}. Then the following hold:

\begin{enumerate}[1)]
	\item First,
	\begin{equation}
\hat{v}^{(k)} = \hat{V}^{(k)} \text{ on } B^+(R) \times ]0,1[.
\end{equation}
Using the change of variables (\ref{coordinatemappings1})-(\ref{coordinatrmapping2}), the $L_{\frac{10}{3}}$ estimate (\ref{Vrescaled10/3uniform}) for $V$, we have that 
\begin{equation}
	\hat{V}^{(k)} \text{ is uniformly bounded in } L_{\frac{10}{3}}(B^{+}(R)\times ]0,1[)
\end{equation}
with bounds independent of $R$.
 In addition, using the local energy estimate~(\ref{Vkuniformenergyest}) for $V$,
\begin{equation}
	\hat{V}^{(k)} \text{ is uniformly bounded in } L_{2,\infty} \cap W^{1,0}_2(B^+(R) \times ]1/4,1[),
\end{equation}
with bounds depending on $R$.

\item Next,
\begin{equation}
	\hat{V}^{(k)} = \sum_{i=0}^3 \hat{V}^{(k),i} \text{ on } B^+(R) \times ]1/4,1[. 
\end{equation}
Using the change of variables (\ref{coordinatemappings1})-(\ref{coordinatrmapping2}) 
and the estimates (\ref{Vk0est})-(\ref{Vk3est}) for $V^{(k),i}$, $i=0,\hdots,3$, 
we obtain the following:
\begin{equation}
	\hat{V}^{(k),0},\hdots,\hat{V}^{(k),3} \text{ are uniformly bounded in } W^{2,1}_{2, \frac{3}{2}}, W^{2,1}_{\frac{9}{8},\frac{3}{2}}, W^{2,1}_{2}, W^{2,1}_{\frac{10}{3}}(B^{+}(R) \times ]1/4, 1[),
\end{equation} 
respectively, 
  with bounds independent of $R$. Moreover, 
$\hat{V}^{(k),0}$ tends to zero in norm.
\item Finally,
\begin{equation}
	\hat{q}^{(k)} = \sum_{i=0}^3 \hat{Q}^{(k),i}.
\end{equation}
Again, using the change of variables (\ref{coordinatemappings1})-(\ref{coordinatrmapping2}) 
and the estimates (\ref{Vk0est})-(\ref{Vk3est}) for $Q^{(k),i}$, $i=0,\hdots,3$, we obtain that
\begin{equation}
	\nabla \hat{Q}^{(k),0},\hdots, \nabla \hat{Q}^{(k),3} \text{ are uniformly bounded in } L_{2, \frac{3}{2}}, L_{\frac{9}{8},\frac{3}{2}}, L_{2}, L_{\frac{10}{3}}(B^{+}(R) \times ]1/4, 1[),
\end{equation}
 respectively, with bounds independent of $R$. Concerning $\nabla \hat{Q}^{(k),1}$, note that
 \begin{equation}
	W^{1,0}_{\frac{9}{8},\frac{3}{2}}(B^+(R) \times ]1/4,1[) \hookrightarrow L_{\frac{3}{2}}(B^+(R) \times ]1/4,1[).
 \end{equation} Moreover, $\nabla \hat{Q}^{(k),0}$ tends to zero in norm.
\end{enumerate}


Let us examine the solution at time $t=1$. 
By~\eqref{eq:disttoL}, we may write $v(\cdot,1) = X + Y$, with $\norm{X}_{L_3^\weak} \leq 2\epsilon$ and $Y \in \mathbb{L}$. Let $X^{(k)}$ and $Y^{(k)}$ denote the rescaled versions of $X$ and $Y$, respectively:
\begin{equation}
	\hat{v}^{(k)}(\cdot,1) = X^{(k)} \circ (\phi_k)^{-1} + Y^{(k)} \circ (\phi_k)^{-1}.
\end{equation}
Since $(\phi_k)^{-1}$ is measure preserving, we may pass to a subsequence such that $X^{(k)} \circ (\phi_k)^{-1} \wstar X^\infty$ in $L_3^\weak$ and $\norm{X^\infty} \leq 2\epsilon$. On the other hand, $Y^{(k)} \circ (\phi_k)^{-1} \wstar 0$ in the sense of distributions on $\R^3_+$:
\begin{equation}
	R_k \left\la Y(R_k \cdot), \zeta \circ \phi_k \right\ra = o(1) \times \norm{\zeta}_{L^{\frac{3}{2},1}} \text{ for all } \zeta \in C^\infty_0(\R^3_+),
\end{equation}
since $\norm{Y^{(k)}}_{L_3^\weak(B(r))} \to 0$ as $r \to 0^+$.


We are ready to pass to the limit. Using the estimate (\ref{almostflat}) for $\nabla^2 \varphi_k$, 1)-3) above, and the compact embedding
\begin{equation}
	W^{2,1}_{s,q}(B^{+}(R) \times ]1/2,1[) \hookrightarrow C([1/2,1]; L_{s}(B^{+}(R)), \quad 1 \leq s \leq \infty, \, q > 1, \, R > 0,
\end{equation}
we can extract a diagonal subsequence that converges in the sense of distributions to a boundary suitable weak solution $(v^{\infty}, q^{\infty})$ of the Navier-Stokes equations in $\mathbb{R}^3_{+}\times ]1/2, 1[$. For all $R>0$,
\begin{equation}
	\label{hatvhatqconv}
	\hat{v}^{(k)} \to v^\infty \text{ in } L_3(B^+(R) \times ]1/2,1[),
\end{equation}
\begin{equation}
	\hat{q}^{(k)} \wto q^\infty \text{ in } L_{\frac{3}{2}}(B^+(R) \times ]1/2,1[),
\end{equation}
see Lemma~\ref{vqcompactnesslemma}.
This and (\ref{singpointflattened}) allow us to apply Proposition~\ref{stabilitysingularpointshalfspace} concerning the stability of singular points 
for the flattened Navier-Stokes equations. Hence, we infer that
\begin{equation}\label{singpointflattenedvinfinity}
(0,1) \text{ is a singular point of } (v^{\infty}, q^{\infty}).
\end{equation}
Furthermore,
\begin{equation}\label{vinfinity10/3}
\norm{v^{\infty}}_{L_{\frac{10}{3}}(\mathbb{R}^3_{+} \times ]1/2, 1[)} \leq C_4(\Omega,M),
\end{equation}
and, using $\mathbf{P}$ to denote $L_{\frac{9}{8},\frac{3}{2}}+L_{2}+L_{\frac{10}{3}}$, we have
\begin{equation}\label{qinfinityspaces}
\norm{\nabla q^{\infty}}_{\mathbf{P}(\mathbb{R}^3_{+} \times ]1/2, 1[)} \leq C_5(\Omega,M).
\end{equation}

\subsubsection*{Step 7: Obtaining the contradiction}

To conclude, we are going to use the following auxiliary Liouville theorem. A similar result was obtained by the authors in~\cite[Remark 4.2]{globalweakbesov} without boundary and employed in~\cite[Theorem 4.1]{albrittonbarkerlocalregI} in a similar manner.
\begin{lemma}[Liouville theorem]
\label{lem:Liouvilletheorem}
	Let $\Omega_\infty = \R^3, \R^3_+$ (with $\Gamma_\infty = \emptyset, \{ x_3 = 0 \}$, respectively) and $(v^\infty,q^\infty)$ be a boundary suitable weak solution in $\Omega_\infty \times ]1/2,1[$ vanishing on $\Gamma_\infty$. For all $M_\infty > 0$, there exist constants $\epsilon_\infty, c_\infty > 0$ depending on $M_\infty$ and satisfying the following property. If
	\begin{equation}
	\label{vinftycompactness}
	\norm{v^\infty}_{L_{\frac{10}{3}}(\Omega_\infty \times ]1/2,1[)} + \norm{\nabla q^\infty}_{\mathbf{P}(\Omega_\infty \times ]1/2,1[)} \leq M_\infty
	\end{equation}
	and
	\begin{equation}
	\norm{v^\infty(\cdot,1)}_{L_3^\weak(\Omega_\infty)} \leq \epsilon_\infty,
	\end{equation}
	then
	\begin{equation}
	|v| \leq c_\infty \text{ on } \Omega_\infty \times ]3/4,1[.
	\end{equation}
\end{lemma}
\begin{proof}[Sketch of proof]
Otherwise, there exists a sequence of solutions on $\Omega_\infty \times ]1/2,1[$ satisfying~\eqref{vinftycompactness} which is converging to zero at $t=1$ and become progressively more singular. After translating in space and passing to a subsequence, 
one obtains a singular boundary suitable weak solution in $\R^3_+ \times ]1/2,1[$ (or $\R^3 \times ]1/2,1[$, depending on $\Omega_\infty$ and the translations). The solution is then controlled at large distances using~\eqref{vinftycompactness} and the $\epsilon$-regularity criterion (here, the pressure is controlled in $L_{\frac{3}{2}}$ on balls, up to its average, by Poincar{\'e}'s inequality). This is enough to apply backward uniqueness. The arguments are similar to those in~\cite[p. 1345-1349]{barkerser16blowup} by Seregin and the second author, for example.
\end{proof}

To apply the Liouville theorem, we set $M_\infty = C_4 + C_5$ from~\eqref{vinfinity10/3}-\eqref{qinfinityspaces} and choose $\epsilon = \epsilon_\infty/2$ and $\Omega^\infty = \R^3_+$.
Uniform bounds 1)-3) in the previous step imply (up to a subsequence) that
\begin{equation}\label{strongconverg}
\hat{v}^{(k)} \rightarrow v^{\infty} \text{ in } C\left([1/2, 1]; L_{\frac{9}{8}}(B^{+}(R))\right) \text{ for all } R>0.
\end{equation}
This means that $v^\infty(\cdot,1) = X^\infty$, where $\norm{X^\infty}_{L_3^\weak} \leq 2\epsilon$. Hence, the hypotheses of Lemma~\ref{lem:Liouvilletheorem} are satisfied, and $v^\infty$ is essentially bounded in $\R^3_+ \times ]3/4,1[$. This contradicts~\eqref{singpointflattenedvinfinity}. 
\end{proof}

\begin{remark}[The interior case]
\label{rmk:interiorcase}
If $x^* \in \Omega$, we do not need to flatten the boundary. Notice that the same \emph{a priori} estimates hold (we derived them before flattening), and the rescaled solutions converge to a suitable weak solution on $\R^3 \times ]1/2,1[$ satisfying~\eqref{vinftycompactness}, small in $L_3^\weak$ at the time $t=1$, and with singularity at the space-time origin. This is enough to apply Lemma~\ref{lem:Liouvilletheorem} and obtain a contradiction.
\end{remark}

\section{Mild bounded ancient solutions}
\label{sec:mbas}

In this section, we will prove Theorem~\ref{thm:mbas}, following the scheme mentioned in the introduction.

To simplify notation, we use the convention that the constants $C$ may implicitly depend on the domain $\Omega$ but are independent of translation, rotation, and rescaling.

\subsection{Truncation and rescaling}
\subsubsection*{Step 1. Truncation procedure}

Once we apply Proposition~\ref{pro:truncation} and rescale appropriately, we have the following situation (where we have replaced $V$ by $v$, for simplicity):
\begin{equation}
	v \in L_\infty(\Omega \times ]-1,t[) \text{ for all } t \in ]-1,0[
\end{equation}
is the unique weak Leray-Hopf solution of the Navier-Stokes equations on $\Omega \times ]-1,0[$ with initial data $v(\cdot,-1)$ and forcing term
\begin{equation}
	 f \in L_{p,\frac{3}{2}}(\Omega \times ]-1,0[) \text{ for all } p \geq 1.
\end{equation}
Furthermore, there exists $x^* \in \Omega \cup \Gamma$ such that
\begin{equation}
	z^* = (x^*,0) \text{ is a singular point of } v
\end{equation}
with the following property. Define
\begin{equation}
	\label{eq:gdef}
	g(t) := \esssup_{-1<t'<t} \norm{u(\cdot,t)}_{L_\infty(\Omega)}, \quad t \in ]-1,0].
\end{equation}
There exists a sequence of points $(z_n)_{n \in \N} = (x_n,t_n)_{n \in \N} \subset \Omega \times ]-1,0[$ with $z_n \to (x^*,0)$ and
\begin{equation}
	1 \leq  M_n := g(t_n) = |v(z_n)| \to \infty \text{ as } n \to \infty.
\end{equation}
In principle, the singular point $z^*$ may be different from the original singular point. If the original (untruncated) solution has an interior singularity, then we may further assume that $x^* \in \Omega$.

\subsubsection*{Step 2. Rescaling procedure}

Consider $\tilde{x_n} \in \p\Omega$ minimizing the distance to $x_n$:
\begin{equation}
	|x_n - \tilde{x_n}| = \dist(x_n,\p\Omega).
\end{equation}
Because $\Omega$ is a bounded $C^3$ domain, there exist a translation and a rotation $\cO_n \in {\rm SO}(3)$ of the original coordinate system\footnote{Translate first and rotate second.} such that, in the new coordinate system, $\tilde{x_n}$ becomes the origin, $\Omega$ becomes $\tilde{\Omega}_n$, and
\begin{equation}
	B(R_0) \cap \tilde{\Omega}_n = \left\lbrace |x| < R_0 : x_3 > \varphi_n(x') \right\rbrace.
\end{equation} 
Here, $\varphi_n \in C^2(\overline{K(R_0)})$ is real-valued function with
\begin{equation}
	\varphi_n(0) = 0, \; \nabla \varphi_n(0) = 0, \text{ and } [\varphi_n]_{C^2(K(R_0))} \leq N_0,
\end{equation}
where
the positive constants $R_0$ and $N_0$ depend only on $\Omega$.  

We rescale about $z_n$ as follows:
\begin{equation}
	v_n(y,s) =  \frac{1}{M_n} \cO_n v \left(\frac{\cO_n^{-1} y}{M_n} +x_n, \frac{s}{M_n^2} +t_n \right),
\end{equation}
where $(y,s) \in Q_n$. Here,
\begin{equation}
	Q_n = \Omega_n \times ]-M_n^2,0[, \quad \Omega_n = M_n \cO_n (\Omega - x_n).
\end{equation}
In the new coordinates, $x_n$ corresponds to $y=0$.
 Moreover,
 \begin{equation}
	|v_n| \leq 1 \text{ on } Q_n,
\end{equation}
 and
\begin{equation}
	|v_n(0)| = 1.
 \end{equation}
By weak-strong uniqueness, $v_n$ is the unique weak Leray-Hopf solution on $Q_n$ with initial data $v_n(\cdot,-M_n^2)$ and forcing term
\begin{equation}
	f_n(y,s) = \frac{1}{M_n^3} \cO_n f \left(\frac{\cO_n^{-1} y}{M_n} +x_n, \frac{s}{M_n^2} +t_n \right). 
\end{equation}
Moreover, $f_n$ converges to zero in certain subcritical norms:
\begin{equation}
	\label{fntozero}
	\norm{f_n}_{L_{p,\frac{3}{2}}(Q_n)} \to 0 \text{ for all } p > \frac{9}{5}.
\end{equation}
We extend $v_n$ by zero to a vector field on $\R^3 \times ]-\infty,0[$. This implies
\begin{equation}
	v_n \wstar U \text{ in } L_\infty(\R^3 \times ]-\infty,0[)
\end{equation}
along a subsequence, for a measurable vector field $U \: \R^3 \times ]-\infty,0[ \to \R^3$.

Let us denote $a_n = \dist(0,\p\Omega_n)$.\footnote{Notice that $a_n = M_n \dist(x_n,\p\Omega)$.} For a subsequence, we have one of the following:

\subsubsection*{Step 2A: Scenario I} 
If
\begin{equation}
	\dist(0,\p\Omega_n) \upto \infty,
\end{equation}
then $\Omega_n \to \R^3$ in a suitable sense. In particular, there exists a subsequence satisfying
\begin{equation}
	B(n) \subset \Omega_n \text{ for all } n \in \N.
\end{equation}


\subsubsection*{Step 2B: Scenario II}
If
\begin{equation}
	\lim_{n \to \infty} \dist(0,\p\Omega_n) = a \geq 0,
\end{equation}
then $\Omega_n \to \R^3_{a}$ in a suitable sense, where
\begin{equation}
	\R^3_{a} := \left\lbrace x \in \R^3 : x_3 > -a \right\rbrace.
\end{equation}
Necessarily, $x^* \in \Gamma$ and $\tilde{x_n} \in \Gamma$ for all sufficiently large $n$. Therefore, for a subsequence,\footnote{Here, we use the following fact concerning bounded $C^2$ domains: There exists a neighborhood $\cN$ of $\p\Omega$ such that for each $x_0 \in \cN \cap \Omega$, there exists a unique $\tilde{x_0} \in \p\Omega$ minimizing $\dist(x_0,\p\Omega)$. Moreover, the vector $\tilde{x_0}-x_0$ is in the direction $\nu(x_0^*)$, where $\nu$ denotes the outer unit normal. Similar statements can be found in \cite[Section 4.4]{krantzimplicit}.
}
\begin{equation}
	\frac{\tilde{x_n}-x_n}{|x_n - \tilde{x_n}|} = \nu(\tilde{x_n}).
\end{equation}
Hence, in the new coordinates, $\tilde{x_n}$ corresponds to $y = -a_ne_3$.

Denote $\tilde{\varphi}_n = M_n \varphi(\cdot/M_n)$. In the new coordinates, whenever $0 < R \leq M_n R_0$,
\begin{equation}
	B(-a_n e_3,R) \cap \Omega_n = \left\lbrace |y+a_ne_3| < R : y_3 + a_n > \tilde{\varphi}_n(y') \right\rbrace.
\end{equation}
  In addition,
 \begin{equation}
	[\tilde{\varphi}_n]_{C^2(K(R))} \leq N_0/M_n.
 \end{equation}
 Consider $R_n = M_n^{\frac{1}{4}} R_0 \to \infty$.\footnote{In the original coordinates, this corresponds to a ball shrinking around the origin, but it is growing in the new coordinates.} By Taylor's theorem,
 \begin{equation}
	|y'| < R_n \text{ implies } |\tilde{\varphi}_n| \leq \frac{N_0}{2 M_n} |y'|^2 \leq \frac{N_0 R_0^2}{2 \sqrt{M_n}}.
 \end{equation}
 Therefore, in a growing ball, $\Omega_n$ contains the region above $\{ y_3 = c_n \}$ for a sequence $c_n \dto a$:\footnote{For example, $c_n = \max(a_n,a) + N_0 R_0^2/(2 \sqrt{M_n})$.}
 \begin{equation}
	E_n := B(-a_ne_3,R_n) \cap \left\lbrace y_3  > c_n \right\rbrace \subset \Omega_n.
 \end{equation}
 Similarly, the complement contains the region below $\{ y_3 = c_n' \}$ for a sequence $c_n' \upto a$:\footnote{$c_n' = \min(a_n,a) - N_0 R_0^2/(2 \sqrt{M_n})$}
 \begin{equation}
	F_n := B(-a_ne_3,R_n) \cap \left\lbrace y_3  < c_n' \right\rbrace \subset \R^3 \setminus \overline{\Omega_n}.
 \end{equation}
There exists a subsequence such that $B(-a_n e_3,R_n) \subset B(-a_{n+1} e_3, R_{n+1})$. Then
	\begin{equation}
	\label{eq:Endef}
	E_n \subset E_{n+1}, \quad \bigcup_{n \in \N} E_n = \R^3_a.
	\end{equation}
	Eventually, we will use~\eqref{eq:Endef} to obtain that $U$ solves the Navier-Stokes equations in $\R^3_a$. 
Also,
\begin{equation}
	F_n \subset F_{n+1}, \quad \bigcup_{n \in \N} F_n = \R^3 \setminus \overline{\R^3_a}.
\end{equation}
Finally, $u_n \equiv 0$ on $F_n \times ]-\infty,0[$ implies
\begin{equation}
	U \equiv 0 \text{ on } \R^3 \setminus \R^3_a.
\end{equation}


\subsection{H{\"o}lder estimates for rescaled truncated solution}

\subsubsection*{Step 3. Showing $U$ is non-trivial (H{\"o}lder estimates)}

In this section, we prove H{\"o}lder estimates for the sequence $(v_n)_{n \in \N}$ described above. Recall that $v_n$ is extended by zero to $\R^3 \times ]-\infty,0[$.

\begin{pro}[H{\"o}lder estimates]
	\label{pro:Holderestimates}
	In the above notation, for all $A>0$,
	\begin{equation}
	\limsup_{n \to \infty} \norm{v_n}_{C^{\frac{1}{2}}_\para(\R^3 \times ]-A,0])} < \infty.
	\end{equation}
\end{pro}

As an immediate corollary, we have
\begin{corollary}[Compactness]
	\label{cor:Holderestimates}
	There exists a subsequence such that
	\begin{equation}
	v_n \to U \text{ in } C(K \times ]-A,0])
	\end{equation}
	for all compact $K \subset \R^3$ and $A > 0$. Hence, $|U(0)| = 1$. In the case of Scenario~II,\footnote{In particular, $a > 0$.}
	\begin{equation}
	U\big|_{\p{\R^3_{a}}}(\cdot,t) = 0 \text{ for all } t \leq 0.
	\end{equation}
\end{corollary}

Let $C_{0,\sigma}(\Omega)$ denote the space of divergence-free vector fields $u_0$ continuous in $\overline{\Omega}$ and vanishing on the boundary. Bootstrapping $v \in L_{\infty}(\Omega \times ]-1,t[)$ via maximal regularity and parabolic Sobolev embedding (see Lemma~\ref{lem:parabolicsobolev}), we obtain that $v \in C([-1,t];C_{0,\sigma}(\Omega))$, for all $t \in ]-1,0[$. Therefore, the following result of Abe~\cite{AbeNSEBounded} will be applicable:

\begin{pro}[NSE with $C_{0,\sigma}$ initial data]
	\label{pro:AbeGiga}
	Let $\Omega \subset \R^3$ be a bounded $C^3$ domain.

	\begin{enumerate}[1.]
	\item For all $u_0 \in C_{0,\sigma}(\Omega)$, there exists $T>0$ satisfying
\begin{equation}
\label{eq:Tlowerbound}
	T \geq \frac{C}{\norm{u_0}_{L_\infty(\Omega)}^2}
\end{equation}
	and a weak solution\footnote{We say nothing here about the pressure, so we mean a weak solution in the sense of divergence-free test functions.} $u \in C(\overline{Q_T})$ of the Navier-Stokes equations in $Q_T$ with initial data $u_0$:
	\begin{equation}
	\label{eq:mildlinftyNSE}
	\left\lbrace
	\begin{aligned}
	\p_t u - \Delta u + \div u \otimes u + \nabla p &= 0 & &\text{ in } Q_T \\
	\div u &= 0 & & \text{ in } Q_T\\
	u\big|_{\p\Omega} &= 0 & &\text{ in } \p\Omega \times ]0,T[ \\
	u(\cdot,0) &= u_0 & &\text{ in } \Omega.
	\end{aligned}
	\right.
	\end{equation}

		\item
	For all $0 < \alpha,\gamma < 1$, the solution $u$ satisfies the estimates\footnote{One may also obtain $\gamma/2$-H{\"o}lder continuity in time for $\nabla u$, but we will not need this here.}
	\begin{equation}
	\label{eq:AbeGigaspaceholder}
	\sup_{0<t<T} \left( \norm{u}_{L_\infty(\Omega)} + t^{\frac{1}{2}}  \norm{\nabla u}_{L_\infty(\Omega)} + t^{\frac{1+\alpha}{2}} [\nabla u]_{C^\alpha(\Omega)} \right) \leq C(\alpha) \norm{u_0}_{L_\infty(\Omega)},
	\end{equation}
	\begin{equation}
	\label{eq:AbeGigatimeholder}
	\sup_{x \in \Omega} [u]_{C^\gamma([B,T])} \leq C(\gamma,T/B) \left(T^{-\gamma} + B^{-\gamma} \right) \norm{u_0}_{L_\infty(\Omega)} \; \text{ for all } 0 < B < T.
	\end{equation}
	\item The solution $u$ is the unique weak Leray-Hopf solution of~\eqref{eq:mildlinftyNSE} in $Q_{T'}$ ($0 < T' \leq T$).
\end{enumerate}
\end{pro}
\begin{proof}
Points 1 and 2 are proven in~\cite[Theorem 1.1]{AbeNSEBounded}, except for the constant in~\eqref{eq:AbeGigatimeholder}, which is contained in the proof of \cite[Proposition 3.5]{AbeNSEBounded}. It is clear that Abe's solution belongs to
\begin{equation}
	C([0,T];L_2(\Omega)) \cap W^{1,0}_2(\Omega \times ]A,T[)
\end{equation}
for all $0 < A < T$. Moreover, it satisfies the energy equality on $\Omega \times [A,T]$ (one may justify the integration by parts computation, or refer to Theorem~1.4.1, p. 272, in~\cite{Sohrbook}). To obtain $\nabla u \in L_2(Q_T)$, we allow $A \to 0^+$ in the energy equality. Finally, Point 3 simply asserts weak-strong uniqueness (see Theorem~1.5.1, p. 276, in~\cite{Sohrbook} for a proof). \end{proof}

In principle, for exterior domains, the constants in the linear estimates required to prove Proposition~\ref{pro:AbeGiga} can depend on the time interval under consideration. For example, the constants in the linear estimates on a \emph{fixed} time interval could become large when zooming out on the domain. This is not the case for bounded domains, in which the semigroup is known to have exponential decay.

\begin{pro}[Perturbed NSE]
\label{pro:perturbedNSE}
	Let $\Omega \subset \R^3$ be a bounded $C^3$ domain.

	Let $p > 3$, $s, s_1 > 1$, and $s_2 > 2$. Let $V \in L_{\infty,s_2}(Q_1)$, $W \in L_{\infty,s_1}(Q_1;\R^{3\times3})$, and $f \in L_{p,s}(Q_1)$.
\begin{enumerate}[1.]

	\item There exists $c_0(\Omega,p,s,s_1,s_2) > 0$ satisfying the following property. If 
	\begin{equation}
	\label{VWfprop}
	\norm{V}_{L_{\infty,s_2}(Q_1)} + \norm{W}_{L_{\infty,s_1}(Q_1)} + \norm{f}_{L_{p,s}(Q_1)} \leq c_0,
	\end{equation}
	then there exists
	a weak solution $w \in C([0,1];L_p(\Omega)) \cap W^{1,0}_{p,2}(Q_1)$ of the following perturbed Navier-Stokes equations:
	\begin{equation}
	\label{eq:perturbedNSE}
	\left\lbrace
	\begin{aligned}
		\p_t w - \Delta w + w \cdot \nabla w + V \cdot \nabla w + w \cdot W + \nabla p &=  f & &\text{ in } Q_1\\
		\div w &= 0 & &\text{ in } Q_1\\
		w\big|_{\p'Q_T} &= 0 & &\text{ on } \p'Q_1.
		\end{aligned}
		\right.
	\end{equation}
	with forcing term $f$, zero Dirichlet conditions (in the sense of trace), and zero initial condition. Here, $\p'Q_1$ denotes the parabolic boundary.
	\item The solution $w$ satisfies the estimate
	\begin{equation}
	\label{westintermsoff}
	\norm{w}_{L_{p,\infty}(Q_1)} + \norm{\nabla w}_{L_{p,2}(Q_1)} \leq C \norm{f}_{L_{p,s}(Q_1)}.
	\end{equation}

	\item The solution $w$ is the unique weak Leray-Hopf solution of \eqref{eq:perturbedNSE} in $Q_{T'}$ ($0 < T' \leq 1$).

\end{enumerate}
\end{pro}

\begin{proof}
The proof is largely routine; we include it for completeness.

For a vector field $g$ on $Q_1$, we use the notation $Lg$ to refer to the unique weak solution $u$ of the Stokes equations
\begin{equation}
\left\lbrace
	\begin{aligned}
	\p_t u - \Delta u + \nabla p &= g & &\text{ in } Q_1 \\
	\div u &= 0 & &\text{ in } Q_1\\
	u &= 0 & &\text{ on } \p'Q_1,
	\end{aligned}
	\right.
\end{equation}
when such a solution exists. Also, we define the function space $X$:
\begin{equation}
	X = C([0,1];L_p(\Omega)) \cap W^{1,0}_{p,2}(Q_1),
\end{equation}
\begin{equation}
	\norm{u}_{X} = \norm{u}_{L_{p,\infty}(Q_1)} + \norm{\nabla u}_{L_{p,2}(Q_1)}.
\end{equation}
We have the following estimates for $L$ in $X$:
\begin{equation}
	\label{Lfest}
	\norm{Lf}_{X} \leq C \norm{f}_{L_{p,s}(Q_1)},
\end{equation}
\begin{equation}
	\label{gest}
	\norm{Lg}_{X} \leq C \min\left( \norm{g}_{L_{p,r}(Q_1)}, \norm{g}_{L_{\frac{p}{2},2}(Q_1)} \right),
\end{equation}
whenever $r>1$.
These may be derived from the smoothing estimates for the Stokes semigroup (see Proposition 20 on p. 183 in~\cite{Hieber2016}, for example) and the Hardy-Littlewood-Sobolev inequality, using the requirement $p > 3$.

Using~\eqref{gest}, we define the bilinear form $B \: X \times X \to X$,
\begin{equation}
	B(u,v) = -L(u \cdot \nabla v),
\end{equation}
\begin{equation}
	\norm{B(u,v)}_{X} \leq C_B \norm{u}_{X} \norm{v}_{X},
\end{equation}
as well as the linear operator $L_{V,W} \: X \to X$,
\begin{equation}
	L_{V,W} g = -L(V \cdot \nabla g + g \cdot W),
\end{equation}
\begin{equation}
	\norm{L_{V,W} g}_{X_T} \leq C_2 \left( \norm{V}_{L_{\infty,s_2}(Q_T)} + \norm{W}_{L_{\infty,s_1}(Q_T)} \right) \norm{g}_{X_T}.
\end{equation}
If we choose $c_0 > 0$ such that $C_2 c_0 \leq \frac{1}{2}$,
then $(I+L_{V,W})^{-1} \: X_T \to X_T$ exists with operator norm $\norm{(I+L_{V,W})^{-1}}_{X_T \to X_T} \leq 2$, provided that~\eqref{VWfprop} is satisfied.

The perturbed Navier-Stokes equations are now equivalent to the integral equation
\begin{equation}
	w = (I+L_{V,W})^{-1} Lf + (I+L_{V,W})^{-1} B(w,w),
\end{equation}
which may be solved by a contraction mapping argument (see \cite[Appendix]{gallagherasymptotics} or \cite[Chapter 5]{bahourichemindanchin}, for example) as long as
\begin{equation}
	\norm{Lf}_{X_T} \leq \frac{c_1}{C_B},
\end{equation}
	where $c_1 > 0$ is a small absolute constant. In light of~\eqref{Lfest}, this is ensured by reducing the size of $c_0$, which completes the existence proof. The relevant contraction also gives~\eqref{westintermsoff}.

	The proof of energy equality and weak-strong uniqueness can be found in the references to Sohr's book~\cite{Sohrbook} mentioned in the proof of Proposition~\ref{pro:AbeGiga}.
\end{proof}

\begin{corollary}[H{\"o}lder continuity for perturbed NSE]
\label{cor:HolderperturbedNSE}
We adopt the hypotheses of Proposition~\ref{pro:perturbedNSE} with the values $p = 12$, $s=s_1=\frac{3}{2}$, and $s_2=2018$.

Suppose additionally that $V, t^{\frac{1}{2}} W \in L_\infty(Q_1)$ and $f \in L_{18,\frac{3}{2}}(Q_1)$. Then the solution $w$ on $Q_1$ from Proposition~\ref{pro:perturbedNSE} satisfies the estimate
	\begin{equation}
[w]_{C^{\frac{1}{2}}(Q_{1})} \leq C\left(\norm{f}_{L_{12,\frac{3}{2}}(Q_1)},\norm{f}_{L_{18,\frac{3}{2}}(Q_1)},\norm{V}_{L_\infty(Q_1)},\norm{t^{\frac{1}{2}}W}_{L_\infty(Q_1)} \right).
	\end{equation}
The constant $C > 0$ is an increasing function of its arguments.
\end{corollary}

The main requirement for the exponents is to choose $\frac{1}{2} < \kappa := 1 - \frac{3}{p} < 1$, since the H{\"o}lder exponent is (at most) $2\kappa-1$, see~\eqref{wcdotnablawterm}. Here, $\kappa = \frac{3}{4}$. As for the remaining indices, the choice $s = \frac{3}{2}$ is natural in our situation, we choose $L_{18,\frac{3}{2}}$ so that~\eqref{fterm} holds, and we can choose any $1 < s_1 < 2$ and $2 < s_2 \leq \infty$ to use Proposition~\ref{pro:perturbedNSE}.

\begin{proof}
	In order to bootstrap, we decompose the solution in $Q_1$ as 
	\begin{equation}
	w = Lf - L(w\cdot \nabla w) - L(V \cdot \nabla w + w \cdot W).
	\end{equation}

	Since $f \in L_{18,\frac{3}{2}}(Q_1)$, maximal regularity and parabolic Sobolev embedding (see Corollary~\ref{cor:parabolicsobolevzerobcs}) into H{\"o}lder spaces imply
	\begin{equation}
	\label{fterm}
	[Lf]_{C^{\frac{1}{2}}_\para(Q_1)} \leq C \norm{Lf}_{\dot W^{2,1}_{18,\frac{3}{2}}(Q_1)} \leq C\norm{f}_{L_{18,\frac{3}{2}}(Q_1)}.
	\end{equation}
	
	Since $w \cdot \nabla w \in L_{6,2}(Q_1)$, by the same arguments and the estimates on $w \cdot \nabla w$ from Proposition~\ref{pro:perturbedNSE},
	\begin{equation}
	\label{wcdotnablawterm}
	[L(u \cdot \nabla u)]_{C^{\frac{1}{2}}_\para(Q_1)} \leq C \norm{L(w \cdot \nabla w)}_{\dot W^{2,1}_{6,2}(Q_1)} \leq C \norm{f}_{L_{12,\frac{3}{2}}(Q_1)}^2.
	\end{equation}

	Finally, by our extra assumptions on $V$ and $W$, we have $V \cdot \nabla w$ and $w \cdot W$ belong to $L_{12,\frac{8}{5}}(Q_1)$. Hence, by similar arguments,
	\begin{equation}
	\begin{aligned}
	\left[L(V \cdot \nabla w + w \cdot W)\right]_{C^{\frac{1}{2}}_\para(Q_1)} &\leq C\norm{L(V \cdot \nabla w + w \cdot W)}_{\dot W^{2,1}_{12,\frac{8}{5}}(Q_1)} \\
	 &\leq C \norm{f}_{L_{12,\frac{3}{2}}(Q_1)} \left( \norm{V}_{L_\infty(Q_1)} + \norm{t^{\frac{1}{2}} W}_{L_\infty(Q_1)} \right).
	\end{aligned}
	\end{equation}
\end{proof}

\begin{lemma}[A decomposition]
\label{lem:Holdercontdecomp}
There exists $T_\sharp \in ]0,1]$ and $\epsilon_0 > 0$ satisfying the following properties. Suppose that $v$ is a weak Leray-Hopf solution on $Q_{T_\sharp}$ satisfying
\begin{equation}
	|v| \leq 1 \text{ on } Q_{T_\sharp}
\end{equation}
with initial condition $v(\cdot,0) \in C_{0,\sigma}(\Omega)$
and forcing term $f$ satisfying
\begin{equation}
	\label{forcinglessthan1}
	\norm{f}_{L_{18,\frac{3}{2}}(Q_{T_\sharp})} \leq \epsilon_0.
\end{equation}
Then $v = u + w$, where $u$ is the weak solution of the Navier-Stokes equations on $Q_{T_\sharp}$ obtained in Proposition~\ref{pro:AbeGiga} with initial data $v(\cdot,0)$ (and $T_\sharp \leq T$, where $T$ is from Proposition~\ref{pro:AbeGiga}), and the remainder $w$ satisfies
\begin{equation}
	\label{eq:westimate}
	\norm{w}_{L_{12,\infty}(Q_{T_{\sharp}})} \leq C \norm{f}_{L_{12,\frac{3}{2}}(Q_{T_\sharp})}.
\end{equation}
Finally, $v$ satisfies the space-time H{\"o}lder estimate
\begin{equation}
	\label{eq:vholderestimate}
	\norm{v}_{C^{\frac{1}{2}}_\para(\Omega \times [T_{\sharp}/2, T_\sharp])} \leq C.
\end{equation}
\end{lemma}
\begin{proof}
First, Proposition~\ref{pro:AbeGiga} guarantees the existence of a unique solution $u$ to the Navier-Stokes equations in $Q_{T}$ with initial data $v(\cdot,0)$ and satisfying various properties detailed therein.

Let $V = u$ and $W = \nabla u$. Since $\norm{V}_{L_\infty(Q_{T})}$ and $\norm{t^{\frac{1}{2}} W}_{L_\infty(Q_{T})} \leq C$, there exists $0 < S \leq T$ such that, by H{\"o}lder's inequality in time,
\begin{equation}
	\norm{V}_{L_{\infty,2018}(Q_{S})} + \norm{W}_{L_{\infty,\frac{3}{2}}(Q_{S})} \leq \frac{c_0}{2},
\end{equation}
where $c_0$ is the constant in Proposition~\ref{pro:perturbedNSE} for $p=12$, $s=s_1=\frac{3}{2}$, and $s_2=2018$. In addition, we take $T_\sharp := \min(S,1)$ and $\epsilon_0 := c_0/2$ in the statement. That is,
\begin{equation}
	\norm{f}_{L_{12,\frac{3}{2}}(Q_{T_\sharp})} \leq \frac{c_0}{2}.
\end{equation}
If necessary, we redefine $V,W,f \equiv 0$ on $\Omega \times ]T_\sharp,1[$. Hence,~\eqref{VWfprop} is satisfied.

Next, we solve the perturbed Navier-Stokes equations on $Q_1$ with zero initial data, forcing term $f$, and coefficients $V$ and $W$, according to Proposition~\ref{pro:perturbedNSE}.  We denote the solution by $w$.

Define $\tilde{v} = u + w$ on $Q_{T_\sharp}$. Then 
\begin{equation}
	\label{eq:tildevspaces}
	u,w,\tilde{v} \in L_{2,\infty}(Q_{T_\sharp}) \cap W^{1,0}_2(Q_{T_\sharp}) \cap C([0,T_{\sharp}];L_{12}(\Omega)).
\end{equation}
Moreover, $\tilde{v}$ is a weak Leray-Hopf solution on $Q_{T_\sharp}$ with initial data $v(\cdot,0)$ and forcing term $f$, since the integration-by-parts computation to obtain energy equality can be justified using~\eqref{eq:tildevspaces}. By weak-strong uniqueness, $v$ as in the statement of Lemma~\ref{lem:Holdercontdecomp} is identical to $\tilde{v}$ on $Q_{T_\sharp}$.

To conclude, the estimate~\eqref{eq:westimate} follows from Proposition~\ref{pro:perturbedNSE}, and H{\"o}lder continuity follows from
\begin{equation}
	\label{uwupperpartofholdernorm}
	[u]_{C^{\frac{1}{2}}(\Omega \times [T_\sharp/2,T_\sharp])}, \; [w]_{C^{\frac{1}{2}}(Q_{T_\sharp})} \leq C
\end{equation}
using Proposition~\ref{pro:AbeGiga} and Proposition~\ref{pro:perturbedNSE}, respectively. Combining~\eqref{uwupperpartofholdernorm} with $|v| \leq 1$ gives~\eqref{eq:vholderestimate}.
\end{proof}

\begin{proof}[Proof of Proposition~\ref{pro:Holderestimates} (H{\"o}lder estimates)]
Let $A > 0$. We employ Lemma~\ref{lem:Holdercontdecomp} and a covering argument. Let $N \in \N$ such that $A_\sharp := A + T_\sharp/2 \leq M_N^2$ and
\begin{equation}
	\norm{f_N}_{L_{18,\frac{3}{2}}(Q_N)} \leq \epsilon_0,
\end{equation}
where $\epsilon_0$ is as in Lemma~\ref{lem:Holdercontdecomp}. This is possible due to~\eqref{fntozero}.
Let $I = t_0 + [0,T_\sharp] \subset [-A_\sharp,0]$ be a closed interval of length $T_\sharp$. Then Lemma~\ref{lem:Holdercontdecomp} implies
\begin{equation}
	\sup_{n \geq N} \norm{v_n}_{C^{\frac{1}{2}}_\para(\Omega_n \times (t_0 + [T_\sharp/2,T_\sharp]))} \leq C.
\end{equation}Since $I$ was arbitrary, we obtain the result by covering $[-A_\sharp,0]$ with intervals $I$ (and extending by zero in space outside $\Omega$).
\end{proof}

Corollary~\ref{cor:Holderestimates} is immediate from Proposition~\ref{pro:Holderestimates} and the compact embeddings of H{\"o}lder spaces. In the case of Scenario~II, recall that $U \equiv 0$ outside $\R^3_{a}$. The no-slip condition follows from this fact and the continuity of $U$ on $\R^3$.

\subsection{Pressure estimates for rescaled truncated solution}

\subsubsection*{Step 4: Scale-invariant pressure estimates}

We now concern ourselves with pressure estimates for the solution $u$ in Proposition~\ref{pro:AbeGiga}. 

\begin{pro}[Pressure estimates]
\label{pro:pressureestimates}
	Let $u$ be the solution obtained in Proposition~\ref{pro:AbeGiga}. Then the associated pressure gradient $\nabla p$ may be decomposed as
	\begin{equation}
	\nabla p = \nabla p_{u \otimes u} + \nabla p_{\rm h},
	\end{equation}
	where 
	\begin{equation}
	p_{u\otimes u} = (-\Delta)^{-1} \div \div (u \otimes u)
	\end{equation}
	and $p_{\rm h}$ is a harmonic function in $\Omega$. In other words, $p_{u \otimes u} = \sum_{i,j=1}^3 \cR_i \cR_j (u_i u_j)$, where $\cR_i$ is the $i$th Riesz transform ($1 \leq i \leq 3$).
	We have
	\begin{equation}
	\label{puotimesuest}
	\sup_{0<t<T} \norm{p_{u \otimes u}}_{\BMO(\R^3)} + t^{\frac{\alpha}{2}} [p_{u\otimes u}]_{C^\alpha(\R^3)} \leq C(\alpha) \norm{u_0}_{L_\infty(\Omega)}^2
	\end{equation}
	for all $0 < \alpha < 1$. Furthermore, $p_{\rm h}$ may be decomposed as
	\begin{equation}
	p_{\rm h} = p_{\rm h}^1 + p_{\rm h}^2,
	\end{equation}
	where $p_{\rm h}^i$ ($i=1,2$) is harmonic in $\Omega$ and satisfies
	\begin{equation}
	\label{p1est}
	\sup_{0<t<T} t^{\frac{1}{2}} \sup_{x \in \Omega} \dist(x,\p\Omega) |\nabla p_{\rm h}^1(x,t)| \leq C \left(\norm{u_0}_{L_\infty(\Omega)} + T^{\frac{1}{2}} \norm{u_0}_{L_\infty(\Omega)}^2 \right).
	\end{equation}
	\begin{equation}
	\label{p2est}
	\sup_{0<t<T} t^{\frac{\alpha}{2}} \sup_{x \in \Omega} \dist(x,\p\Omega)^{1-\alpha} |\nabla p_{\rm h}^2(x,t)| \leq C(\alpha) \norm{u_0}_{L_\infty(\Omega)}^2,
	\end{equation}
	for all $0 < \alpha < 1$.
\end{pro}

We will adopt the notation $\bP \: L_2(\Omega) \to L_{2,\sigma}(\Omega)$ for the Leray projection obtained from the Helmholtz decomposition, and $\bQ = I - \bP$.

\begin{proof}
We decompose the pressure gradient as
\begin{equation}
	\nabla p = \nabla \Phi + \nabla \pi,
\end{equation}
\begin{equation}
	\nabla \Phi = -\bQ( u \cdot \nabla u ), \quad \nabla \pi = \bQ(\Delta u).
\end{equation}
Observe that $\nabla \pi$ is the pressure gradient associated to the solution $w$ of the Stokes equations
\begin{equation}
\label{eq:Stokesw}
\left\lbrace
\begin{aligned}
	\p_t w - \Delta w + \nabla \pi &= -\bP \div F & &\text{ in } Q_T \\
	\div w &= 0 & &\text{ in } Q_T\\
	w &= 0 & &\text{ on } \p\Omega \times ]0,T[ \\
	w(\cdot,0) &= u_0 & &\text{ in } \Omega,
	\end{aligned}
	\right.
\end{equation}
where $F = u \otimes u$.
By uniqueness for the Stokes equations, $w = u$. The pressure $\pi$ is an effect of the boundary that accounts for the fact that $\Delta w \cdot n\big|_{\p\Omega}$ does not generally vanish, even though the forcing term $- \bP (u \cdot \nabla u)$ is already projected. In other words, $\bP$ does not typically commute with the Laplacian in a bounded domain.  Therefore, it is natural to use estimates which isolate the boundary effects via a weight $\dist(x,\p\Omega)$.


\begin{lemma}
\label{lem:subpressurelem}
Let $u_0 \in C_{0,\sigma}(\Omega)$ and $F \: \Omega \times \R_+ \to \R^{3\times3}$ satisfying $F, t^{\frac{1}{2}} \nabla F \in L_\infty(\R_+;C_0({\Omega}))$. Then the solution $(w,\nabla \pi)$ of the Stokes equations~\eqref{eq:Stokesw} with initial data $u_0$ and forcing term $-\bP \div F$
 satisfies, for all $t>0$,
\begin{equation}
	\label{eq:subpressurelemest}
	t^{\frac{1}{2}} \sup_{x \in \Omega} {\rm dist}(x,\p\Omega) |\nabla \pi(x)| \leq C(\Omega) \left( \norm{u_0}_{L_\infty(\Omega)} + t^{\frac{1}{2}} \norm{F}_{L_\infty(\Omega \times \R_+)}^{\frac{1}{2}} \norm{s^{\frac{1}{2}} \nabla F}_{L_\infty(\Omega \times \R_+)}^{\frac{1}{2}} \right).
\end{equation}
\end{lemma}
Here, $C_0(\Omega)$ denotes the space of continuous functions in $\overline{\Omega}$ vanishing on $\p\Omega$, and $C^1_0(\Omega)$ denotes the space of functions $F$ which are $C^1$ in $\overline{\Omega}$ with $F$ and $\nabla F$ vanishing on $\p\Omega$. The space $C_{0,\sigma}(\Omega)$ consists of divergence-free $C_0(\Omega)$ vector fields.

In order to apply Lemma~\ref{lem:subpressurelem}, we notice that, when $F = u \otimes u$ with $u, t^{\frac{1}{2}} \nabla u\in L_\infty(\R_+;C(\overline{\Omega}))$ and $u|_{\p\Omega}(\cdot,t) = 0$, certainly $F,t^{\frac{1}{2}} \nabla F \in L_\infty(\R_+;C_0(\Omega))$. We extend $F$ forward-in-time by zero if necessary.
\begin{proof}[Proof of Lemma~\ref{lem:subpressurelem}]
Let $\left( S(t) \right)_{t \geq 0}$ denote the Stokes semigroup in $\Omega$ with Dirichlet boundary conditions.
We have the representation formula
\begin{equation}
	\label{wrepformula}
	w = S(t) u_0 - \int_0^t S(t-s) \bP \div F \, ds \; \text{ for all } t \geq 0.
\end{equation}
The following gradient estimate for the semigroup $S(\cdot)$ was proven in~\cite{AbeGiga}:
\begin{equation}
	 t^{\frac{1}{2}} \norm{\nabla S(t) u_0}_{L_\infty(\Omega)} \leq C \norm{u_0}_{L_\infty(\Omega)} \text{ for all } t > 0. 
\end{equation}
In addition, $\nabla S(t) u_0 \in C(\overline{\Omega})$.
Similarly, the following gradient estimate for the composition of operators $S(\cdot)\bP \div$ was proven in~\cite{abestokesflow}:
\begin{equation}
	\label{eq:Gest}
	t^{\frac{3}{4}} \norm{\nabla S(t) \bP \div G}_{L_\infty(\Omega)}  \leq C\norm{G}_{L_\infty(\Omega)}^{\frac{1}{2}} \norm{\nabla G}_{L_\infty(\Omega)}^{\frac{1}{2}} \text{ for all } t > 0 
\end{equation}
for $G \in C^1_0(\Omega)$, and $\nabla S(t) \bP \div G \in C(\overline{\Omega})$.
The estimate~\eqref{eq:Gest} implies
\begin{equation}
\begin{aligned}
	\left\lVert \nabla \int_0^t S(t-s) \bP \div F \right\rVert_{L_\infty(\Omega)} &\leq C \int_0^t (t-s)^{-\frac{3}{4}} s^{-\frac{1}{4}} \norm{F(\cdot,s)}_{L_\infty(\Omega)}^{\frac{1}{2}} \norm{s^{\frac{1}{2}} \nabla F(\cdot,s)}_{L_\infty(\Omega)}^{\frac{1}{2}} \, ds \\ &\leq C \norm{F}_{L_\infty(\Omega \times \R_+)}^{\frac{1}{2}} \norm{s^{\frac{1}{2}} \nabla F}_{L_\infty(\Omega \times \R_+)}^{\frac{1}{2}}.
	\end{aligned}
\end{equation}
Hence, by the representation formula~\eqref{wrepformula} and the above estimates, $\nabla w(\cdot,t) \in C(\overline{\Omega})$, and
\begin{equation}
	t^{\frac{1}{2}} \norm{\nabla w(\cdot,t)}_{L_\infty(\Omega)} \leq C \left( \norm{u_0}_{L_\infty} + t^{\frac{1}{2}} \norm{F}_{L_\infty(\Omega \times \R_+)}^{\frac{1}{2}} \norm{s^{\frac{1}{2}}\nabla F}_{L_\infty(\Omega \times \R_+)}^{\frac{1}{2}} \right),
\end{equation}
for all $t \in \R_+$.
Since $\nabla \pi = \bQ(\Delta w)$, the proof is completed by applying Lemma~\ref{lem:ellipticest} and Lemma~\ref{cor:pressurecor1}. (In order to apply Lemma~\ref{cor:pressurecor1}, we use maximal $L_2$ regularity to obtain that $w(\cdot,t) \in H^2(\Omega)$ for a.e. $t \in \R_+$. By weak continuity in time, the resulting estimate is valid for all $t \in \R_+$.)
\end{proof}

We now continue with the proof of Proposition~\ref{pro:pressureestimates}.

First, we apply Lemma~\ref{lem:subpressurelem} to estimate on $p^1_{\rm h}$. Specifically, we consider $F = u \otimes u$ and $p^1_{\rm h} = \pi$. Combining~\eqref{eq:subpressurelemest} with the estimate~\eqref{eq:AbeGigaspaceholder} for $u$ and $\nabla u$ in Proposition~\ref{pro:AbeGiga} yields~\eqref{p1est}.

Next, we are going to deal with $\nabla \Phi = -\bQ(u \cdot \nabla u)$. We use the method in \cite[Lemma 3.3]{abestokesflow}. Specifically, we decompose $\Phi = \Phi_1 + \Phi_2$, where
\begin{equation}
	\Phi_1 = (-\Delta)^{-1} \div \div(u \otimes u),
\end{equation}
and we consider $u(\cdot,t)$ as a function on $\R^3$.
Notice that, for all $0 < t < T$,
\begin{equation}
	\label{Phi1bmoest}
	\norm{\Phi_1(\cdot,t)}_{\BMO(\R^3)} \leq C \norm{u(\cdot,t)}_{L_\infty(\Omega)}^2 \leq C\norm{u_0}_{L_\infty(\Omega)}^2,
\end{equation}
\begin{equation}
	\label{Phi1alphaest}
	t^{\frac{\alpha}{2}} [\Phi_1(\cdot,t)]_{C^\alpha(\R^3)} \leq C t^{\frac{\alpha}{2}} \norm{u(\cdot,t)}_{L_\infty(\Omega)} [u(\cdot,t)]_{C^\alpha(\Omega)} \leq C\norm{u_0}_{L_\infty(\Omega)}^2,
\end{equation}
by the estimates in Proposition~\ref{pro:AbeGiga}. We define $p_{u \otimes u} = \Phi_1$ to obtain~\eqref{puotimesuest} from~\eqref{Phi1bmoest}-\eqref{Phi1alphaest}. Finally, Lemma~\ref{lem:ellipticest}, Lemma~\ref{cor:pressurecor2} with $F = u \otimes u$, and~\eqref{Phi1alphaest} imply
\begin{equation}
	t^{\frac{\alpha}{2}} \sup_{x \in \Omega} \dist(x,\p\Omega)^{1-\alpha} |\nabla \Phi_2(x,t)| \leq C t^{\frac{\alpha}{2}} [\Phi_1(\cdot,t)]_{C^\alpha(\R^3)} \leq C \norm{u_0}_{L_\infty(\Omega)}^2.
\end{equation}
We define $p_{\rm h}^2 = \Phi_2$ to obtain~\eqref{p2est} and complete the proof. \end{proof}

\subsection{Conclusion}

\subsubsection*{Step 5: Showing $U$ is a mild solution}

We now adopt the notation from the beginning of Section~\ref{sec:mbas} and complete the proof of Theorem~\ref{thm:mbas}. 

Let $I = t_0 + [0,T_\sharp] \subset ]-\infty,0]$ be a closed interval, where $T_\sharp$ is as in Lemma~\ref{lem:Holdercontdecomp}. Denote $I/2 = t_0 + [T_\sharp/2,T_\sharp]$, $\mathring{I}$ the interior of $I$, etc. There exists $N \in \N$ such that for all $n \geq N$, we may decompose $v_n$ on $\Omega_n \times I$ as in Lemma~\ref{lem:Holdercontdecomp}. That is, 
\begin{equation}
	v_n = u_n + w_n \text{ on } \Omega_n \times I,
\end{equation}
where $u_n$ is the solution obtained in Proposition~\ref{pro:AbeGiga} satisfying
\begin{equation}
	\sup_{n \geq N} \norm{u_n}_{C^\alpha_{\para}(\R^3 \times I/2)} < \infty,
\end{equation}
and $w_n$ is a perturbation accounting for the forcing term $f_n$:
\begin{equation}
	\label{eq:wntozero}
	\norm{w_n}_{L_{12,\infty}(\R^n \times I)} \leq C\norm{f_n}_{L_{12,\frac{3}{2}}(Q_n)} \dto 0,
\end{equation}
where $u_n$ and $w_n$ are extended by zero to $\R^3 \times I$.
Hence, up to a subsequence,
\begin{equation}
	u_n \to U \text{ in } C(K \times I/2)
\end{equation}
for all compact $K \subset \R^3$. Here, we have used~\eqref{eq:wntozero} to conclude that 
\begin{equation}
	\lim_{n \to \infty} u_n = \lim_{n \to \infty} v_n \; (= U) \; \text{ on } \R^3 \times I.
\end{equation}
It suffices to analyze the solution $u_n$ and its associated pressure $p_n$, and we will no longer deal with~$w_n$.

\subsubsection*{Step 5A: Scenario I}
Let us recall the pressure estimates for $\nabla p_n$ obtained in Proposition~\ref{pro:pressureestimates}. To begin,
\begin{equation}
	\label{eq:pnuotimesuweakconv}
	(p_n)_{u \otimes u} \wstar P_{U \otimes U} \text{ in } L^\infty_t \BMO_x(\R^3 \times I/2),
\end{equation}
and it is possible to prove that $P_{U \otimes U} = (-\Delta)^{-1} \div \div U \otimes U$.
 
 In Scenario~I, $\dist(\cdot,\p\Omega_n) \to \infty \text{ in } C(K)$ for compact $K \subset \R^3$. Hence,
\begin{equation}
	\nabla (p_n)_{\rm h} \to 0 \text{ in } L_{\infty}(K \times I/2) \text{ for all compact } K \subset \R^3.
\end{equation}
Therefore, we have\footnote{We interpet the limit in the sense that for each $R > 0$, there exists $N \in \N$ such that for all $n \geq N$, $\nabla p_n$ is well defined on $B(R) \times I/2$.}
\begin{equation}
	\nabla p_n \wstar \nabla P \; (= \nabla P_{U \otimes U}) \text{ in } \mathcal{D}'(B(R) \times \mathring{I}/2) \text{ for all } R > 0.
\end{equation}
Since also $u_n \to U$ in $C(K \times I/2)$ for all compact $K \subset \R^3$, we obtain that $(U,\nabla P)$ is a weak solution of the Navier-Stokes equations in $\R^3 \times \mathring{I}/2$. Since $I$ was arbitrary and our estimates were independent of $I$, we obtain that $U$ is a bounded ancient solution of the Navier-Stokes equations in $\R^3$. Finally, since $P \in L^\infty_t \BMO_x(\R^3 \times ]-\infty,0[)$, the equivalent characterization of mild bounded solutions in $\R^3$ in~\cite{barkersereginmbas} implies that $U$ is a mild bounded ancient solution in $\R^3$. \qed

\subsubsection*{Step 5B: Scenario II}
Regarding the convergence of $(p_n)_{u \otimes u}$, we similarly have~\eqref{eq:pnuotimesuweakconv}. The main difference in Scenario~II concerns the harmonic pressure. Observe
\begin{equation}
	\dist(\cdot,\p\Omega_n) \to (x_3+a) \text{ in } C(K)
\end{equation}
for all compact $K \subset \R^3$. Hence, by the pressure estimates in Proposition~\ref{pro:pressureestimates}, there exists a subsequence such that
\begin{equation}
	\nabla (p_n)_{\rm h}^i \wstar \nabla P_{\rm h}^i \text{ in } L_\infty(K \times I/2) \text{ for all compact } K \subset \R^3_a, \; i=1,2,
\end{equation}
where $P_{\rm h}^i$ is harmonic ($i=1,2$).
Moreover, we retain the weighted estimates
\begin{equation}
	\esssup_{t \in I/2} \sup_{x \in \R^3_a} (x_3+a) |\nabla P_{\rm h}^1| \leq C,
\end{equation}
\begin{equation}
	\esssup_{t \in I/2} \sup_{x \in \R^3_a} (x_3+a)^{\frac{1}{2}} |\nabla P_{\rm h}^2| \leq C.
\end{equation}
Define $P = P_{U \otimes U} + \sum_{i=1,2} P_{\rm h}^i$.

Recall the increasing sequence of open sets in~\eqref{eq:Endef}. By the strong convergence $u_n \to U$ and weak-$\ast$ convergence $\nabla p_n \wstar \nabla P$ in $E_N \times \mathring{I}/2$ for each $N \in \N$, we obtain that $(U,\nabla P)$ is a weak solution of the Navier-Stokes equations in $\R^3_a \times \mathring{I}/2$. Since $I$ was arbitrary and our estimates independent of $I$, we obtain that $U$ is a bounded ancient solution of the Navier-Stokes equations in $\R^3_{a}$ with pressure gradient $\nabla P$.

It remains to prove that $U$ is a mild solution. Again, we use a characterization in terms of the pressure. In the half space, it is more convenient to use the decomposition
\begin{equation}
\label{halfspacepressuredecomp}
	P = P_{U \otimes U}^+ + P_{\rm h}^+,
\end{equation}
where $P_{U \otimes U}^+$ is the solution of the boundary value problem
\begin{equation}
\left\lbrace
\begin{aligned}
	-\Delta P_{U \otimes U}^+ &= \div \div U \otimes U \text{ in } \R^3_{a} \\ \frac{\p P_{U \otimes U}^+}{\p x_3} &= 0 \text{ on } \p\R^3_{a}
\end{aligned}
\right.
\end{equation}
obtained by reflecting $U \otimes U\big|_{\{x_3 > -a \}}$ evenly across the plane $\{ x_3 = -a \}$ and solving on the whole space. Hence,
\begin{equation}
	P_{U \otimes U}^+ \in L^\infty_t \BMO_x(\R^3_{a} \times ]-\infty,0[).
\end{equation}
Then $\nabla P_{\rm h}^+ = \sum_{i=1}^3 \nabla P_{\rm h}^i$,
where $P_{\rm h}^3 = P_{U \otimes U} - P_{U \otimes U}^+$ on $\R^3_{a}$.
Moreover,
\begin{equation}
	P_{\rm h}^3 \in L^\infty_t \BMO_x(\R^3_{a} \times ]-\infty,0[),
\end{equation}
and by gradient estimates for harmonic functions,
\begin{equation}
	|\nabla P_{\rm h}^3(x,t)| \leq \frac{C}{x_3} \fint_{B(x,\frac{x_3}{2})} \left| P_{\rm h}^3 - [P_{\rm h}^3]_{B(x,\frac{x_3}{2})} \right| \, dy  \leq \frac{C}{x_3} \norm{P_{\rm h}^3}_{L^\infty_t \BMO_x(\R^3_{a} \times ]-\infty,0[)},
\end{equation}
for all $x_3 > 0$ and almost every $t < 0$.

To summarize, $U \in C(\overline{\R^3_{a}} \times ]-\infty,0[)$ is bounded in $\R^3_{a} \times ]-\infty,0[$ and solves the Navier-Stokes equations in $\R^3_{a}$ with no-slip boundary condition. Its pressure satisfies~\eqref{halfspacepressuredecomp}, and
\begin{equation}
\label{eq:halfspacepressuredecay}
	|\nabla P_{\rm h}^+(x,t)| \leq C(t) \log \left( 2+\frac{1}{x_3} \right) \text{ whenever } x_3 > 1 \text{ and } t < 0.
\end{equation} Hence, by the characterization of mild bounded ancient solutions in~\cite{barkersereginmbas}, $U$ is a mild bounded ancient solution in $\R^3_{a}$. \qed

\begin{remark}
Technically, the requirement in~\cite{barkersereginmbas} is that \eqref{eq:halfspacepressuredecay} is satisfied for all $x_3 > 0$ (when $a = 0$). However, a careful inspection of the proof shows that the behavior near $x_3 = 0$ is not important. In fact, a requirement of the form $\nabla' P \to 0$ as $x_3 \to \infty$ is enough to rule out parasitic solutions, see~\cite[Theorem 5]{prangeliouville}.
\end{remark}

\begin{appendix}
\numberwithin{theorem}{section}

\section{Persistence of singularities}

\label{sec:persistenceofsing}

In this appendix, we recall certain facts related to boundary suitable weak solutions of the flattened Navier-Stokes equations. Our main goal is to prove the persistence of singularities\footnote{It is also sometimes called propagation or stability of singularities.} lemma near a curved boundary in~Proposition~\ref{stabilitysingularpointshalfspace}.

Previously, such stability properties have been established for interior singular points by Rusin and {\v S}ver{\'a}k in~\cite{rusinsver}. In~\cite{rusinsver} and the paper of Jia and {\v S}ver{\'a}k~\cite{Jiasver2013}, persistence of singularities was used to show existence of minimal blow-up  $\dot{H}^{\frac{1}{2}}$ and $L_{3}$ initial data for the three dimensional Navier-Stokes equations in the whole space. The authors adapted this approach to critical Besov spaces in~\cite{globalweakbesov}. The analogous stability lemma was later established for boundary singular points of the Navier-Stokes equation by the second author in his thesis~\cite{Barkerthesis}. See the thesis~\cite{Tuanthesis} of Pham for results related to minimal blow-up data in the half-space.

For the regularity theory of the Navier-Stokes equations against curved boundaries, our main resources are~\cite{sereginshilkinsolonnikovboundarypartialreg2014,mikhailovshilkincurvedboundary2006,sereginshilkinsurvey}. These works generalize the analogous theory for the flat boundaries developed in~\cite{seregincknflatboundary,sereginl3inftyflatboundary}. 




As in~\cite{sereginshilkinsolonnikovboundarypartialreg2014} and~\cite{mikhailovshilkincurvedboundary2006}, for $\varphi \in C^{2}(\overline{K(R)})$, we define the operators
\begin{equation}
	\hat{\nabla}= \hat{\nabla}_{\varphi}:=(\frac{\partial}{\partial x_1}-\frac{\partial \varphi}{\partial x_1}\frac{\partial}{\partial x_3}, \frac{\partial}{\partial x_2}-\frac{\partial \varphi}{\partial x_2}\frac{\partial}{\partial x_3}, \frac{\partial}{\partial x_3})
\end{equation}
and, with summation over repeated indices,
\begin{equation}
	\hat{\Delta}= \hat{\Delta}_{\varphi}:=a_{ij}(x) \frac{\partial^{2}}{\partial x_{i}\partial x_{j}}+ b_{i}(x)\frac{\partial}{\partial x_{i}}.
\end{equation}
Here,
\begin{equation}
	 a_{11}= a_{22}=1, \quad a_{33}(x)= 1+\left(\frac{\partial \varphi}{\partial x_{1}}\right)^2+\left(\frac{\partial \varphi}{\partial x_{2}}\right)^2, 
\end{equation}
\begin{equation}
	 a_{12}=a_{21}=0, \quad a_{13}=a_{31}= -\frac{\partial \varphi}{\partial x_{1}}, \quad a_{23}=a_{32}=-\frac{\partial \varphi}{\partial x_{2}}, 
\end{equation}
\begin{equation}
	 b_1=b_2=0, \quad b_3= -\frac{\partial^{2} \varphi}{\partial x_1^{2}}-\frac{\partial^{2} \varphi}{\partial x_2 ^{2}}.
\end{equation}

\begin{definition}[Boundary suitable weak solution of the flattened NSE]
\label{swsflatten}
Let $R>0$. We say that $(v,p,\varphi)$ is a \emph{boundary suitable weak solution of the flattened Navier-Stokes equations} in $Q^+(R)$ if the following conditions are satisfied:
\begin{enumerate}
\item $v\in L_{2,\infty} \cap W^{1,0}_{2} \cap W^{2,1}_{\frac{9}{8}, \frac{3}{2}}(Q^+(R))$, $p\in L_{\frac{3}{2}}\cap W^{1,0}_{\frac{9}{8}, \frac{3}{2}}(Q^{+}(R))$, and $\varphi \in C^{2}(\overline{K(R)})$,
\item $(v,p,\varphi)$ solve the flattened Navier-Stokes equations in the sense of distributions on $Q^+(R)$:
\begin{equation}
\left\lbrace
\begin{aligned}
	\frac{\partial v}{\partial t}+\hat{\nabla}_{\varphi}\cdot( v\otimes v)-\hat{\Delta}_{\varphi}v+ \hat{\nabla}_{\varphi}p&=0 & &\text{ in } Q^+(R) \\
	\hat{\nabla}_{\varphi}\cdot v&=0 & &\text{ in } Q^+(R) \\
	 v\big|_{x_3=0} &=0 & &\text{ on } \{ x_3 = 0 \} \times ]-R^2,0[
	\end{aligned}
	\right.
\end{equation}
with boundary condition in the trace sense, and
\item $(v,p, \varphi)$ satisfy the local energy inequality:
\begin{equation}\label{localenergyinequalityflattened}
\int_{B^{+}(R)} \zeta(x,t) |v(x,t)|^2 dx+2\int_{-R^2}^{t} \int_{B^{+}(R)} \zeta |\hat{\nabla}_{\varphi} v|^2 dxdt'\leq $$$$ 
\int_{-R^2}^{t} \int_{B^{+}(R)} |v|^2\left( \frac{\partial \zeta}{\partial t}+\hat{\Delta}_{\varphi} \zeta\right)+ v\cdot \hat{\nabla}_{\varphi} \zeta(|v|^2+2p) \, dx \, dt'
\end{equation}
for almost every $t\in ]-R^2,0[$ and all non-negative functions $\zeta \in C^{\infty}_{0}(B(R) \times ]-R^2,0])$.
\end{enumerate}
\end{definition}
In~\cite{sereginshilkinsolonnikovboundarypartialreg2014} and other papers, what we term the flattened Navier-Stokes equations are referred to as the perturbed Navier-Stokes equations. However, we use the term ``perturbed" to refer to the inclusion of lower order terms.

The following lemma is more-or-less standard and follows from the local energy inequality~\eqref{localenergyinequalityflattened} with an appropriate choice of test function and the Aubin-Lions lemma.
\begin{lemma}[Compactness]
\label{vqcompactnesslemma}
Let $(v^{(k)},q^{(k)},\varphi^{(k)})_{k \in \N}$ be a sequence of boundary suitable weak solutions of the flattened Navier-Stokes equations in $Q^+$ satisfying
\begin{equation}
	\sup_{k \in \N} \norm{v^{(k)}}_{L_3(Q^+(R))} + \norm{q^{(k)}}_{L_{\frac{3}{2}}(Q^+(R))} \leq M,
\end{equation}
\begin{equation}
	\varphi^{(k)} \to \varphi \text{ in } C^2(\overline{K(R)}).
\end{equation}
Then there exists $(v,q,\varphi)$ a boundary suitable weak solution of the flattened Navier-Stokes equations on $Q^+(R)$, for every $0 < R < 1$, with
\begin{equation}
	v^{(k)} \to v \text{ in } L_3(Q^+(R)),
\end{equation}
\begin{equation}
	q^{(k)} \wto q \text{ in } L_{\frac{3}{2}}(Q^+(R)).
\end{equation}
\end{lemma}

In addition, we often impose the additional conditions
\begin{equation}
\label{eq:additionalphiconditions}
	 \varphi(0)=1, \; \nabla \varphi(0)=0, \text{ and } [\varphi]_{C^2(K(R))} \leq \frac{\mu_*}{2R},
\end{equation}
where $\mu^{*} > 0$ is a small positive constant defined in \cite[Lemma 3.1]{sereginshilkinsolonnikovboundarypartialreg2014} whose exact value is not important for us. Essentially, $0 < \mu^{*} \ll 1$ ensures maximal regularity estimates for the linear equation
\begin{equation}
\left\lbrace
\begin{aligned}
	\frac{\partial v}{\partial t}-\hat{\Delta}_{\varphi}v+\hat{\nabla}_{\varphi} q &= f & &\text{ in } \mathbb{R}^3_{+} \times \R_+ \\
	\hat{\nabla}_{\varphi} \cdot v&=0 & &\text{ in } \mathbb{R}^3_{+} \times \R_+ \\
	v|_{x_3=0} &= 0 & &\text{ in } \{x_3 = 0\} \times \R_+ \\
	v(\cdot,0) &= 0 & &\text{ in } \R^3_+.
	\end{aligned}
	\right.
\end{equation}
One can prove these estimates by perturbing around the solution with $\varphi = 0$ and estimating $(\Delta - \hat{\Delta}) v$, $(\nabla - \hat{\nabla}) \cdot v$, and $(\nabla - \hat{\nabla}) q$. This also requires treating non-zero divergence. This perturbation argument is not `semilinear in nature'; it requires the full maximal regularity for the half-space in order to conclude.

In the following, we will adopt the notation\footnote{In the literature, $C$ is sometimes defined as $C^3$ instead of $C$ in~\eqref{velocityscaleinvariant}, and similarly for $D$ and $D^{\frac{3}{2}}$ in~\eqref{pressurescaleinvariant}.}
\begin{equation}\label{velocityscaleinvariant}
C(v,R):=\left(\frac{1}{R^2} \int_{Q^{+}(R)} |v|^3 \, dx\, dt\right)^{\frac{1}{3}}
\end{equation}
and
\begin{equation}\label{pressurescaleinvariant}
D(q,R):= \left( \frac{1}{R^2} \int_{Q^{+}(R)} |q-[q]_{B^{+}(R)}|^{\frac{3}{2}} \, dx \,dt\right)^{\frac{2}{3}}.
\end{equation}
We will omit the dependence on $v$ and $q$ when it is clear from context.

 \begin{theorem}[$\epsilon$-regularity, Theorem 4.1 in \cite{sereginshilkinsolonnikovboundarypartialreg2014}]\label{partialregflattened}
 There exist absolute constants $\epsilon_{*},c_*>0$ satisfying the following property. Let $(v,q,\varphi)$ be a boundary suitable weak solution of the flattened Navier-Stokes equations in $Q^+$ satisfying~\eqref{eq:additionalphiconditions}. If
 \begin{equation}\label{epsilonregcondition}
 C(1)+D(1)<\epsilon_{*}
 \end{equation}
 then $v$ is H\"{o}lder continuous in $\overline{Q^{+}(1/2)}$, and
 \begin{equation}
	\sup_{{Q}^{+}(1/2)} |v| \leq c_*.
 \end{equation}
 \end{theorem}

 \begin{pro}[Pressure decay estimate, p. 2930 of \cite{mikhailovshilkincurvedboundary2006}]
 \label{scaleinvariantestimatesflattened}
 Let $(v,q,\varphi)$ be a boundary suitable weak solution of the flattened Navier-Stokes equation in $Q^{+}$ satisfying~\eqref{eq:additionalphiconditions}. Then, for any $\rho\in ]0,1[$ and $\theta\in ]0,\frac{1}{2}[$, we have
 \begin{equation}\label{pressureestimate}
 D(\theta \rho)\leq [c_3\theta^{\frac{4}{3}}+c_4\theta^{-1} C(\rho)]D(\rho)+ c_5 \theta^{\frac{4}{3}}[C(\rho)+C^{\frac{3}{2}}(\rho)]+c_6 \theta^{-1}[C^2(\rho)+C^3(\rho)].
 \end{equation} 
 Here, $c_3$-$c_6$ are universal positive constants.
 \end{pro}

 Let us now present the main proposition. 
 \begin{pro}[Persistence of singularities]
 \label{stabilitysingularpointshalfspace}
Let $(v^{(k)},p^{(k)},\varphi^{(k)})_{k \in \N}$ 
 be a sequence of boundary suitable weak solutions to the flattened Navier-Stokes equations in $Q^{+}$ with $\varphi^{(k)}$ satisfying~\eqref{eq:additionalphiconditions}.
If
\begin{equation}\label{strongconverguk}
v^{(k)} \to v \text{ in } L_{3}(Q^+),
\end{equation}
\begin{equation}\label{weakconvergpressure}
p^{(k)} \wto p \text{ in } L_{\frac{3}{2}}(Q^{+}),
\end{equation}
and
\begin{equation}
	\label{eq:limsupvk}
	\limsup_{k \to \infty} \norm{v^{(k)}}_{L_\infty(B^+(R))} = \infty \text{ for all } 0 < R < 1,
\end{equation}
 Then
 \begin{equation}
 	\label{eq:vsingpoint}
	v \text{ has a singular point at the space-time origin}.
 \end{equation}
\end{pro}

The arguments are based on those in the the second author's thesis~\cite{Barkerthesis} for flat boundaries, which in turn essentially follow arguments in Seregin's paper~\cite{sereginl3inftyflatboundary}.
Note that one may remove the smallness condition~\eqref{eq:additionalphiconditions} by zooming in if,  for example, $\varphi^{(k)} \to \varphi$ in $C^2(K)$.

\begin{proof}[Proof of Proposition~\ref{stabilitysingularpointshalfspace}]

We prove the contrapositive, i.e., the failure of~\eqref{eq:vsingpoint} implies the failure of~\eqref{eq:limsupvk}. Suppose there exists $0 < R_0 \leq 1$ with
\begin{equation}\label{ubounded}
v\in L_{\infty}(Q^{+}(R_0)).
\end{equation}
By zooming in, we may assume that $R_0 = 1$. The first step is to deal with $C(R)$.
For any $0 < R  \leq 1$,
\begin{equation}\label{scaleinvariantu}
C(v,R)^3 = \frac{1}{R^2} \int_{Q^+(R)} |v|^3 \,dx \,dt \leq R^3 \norm{v}_{L_{\infty}(Q^{+})}^3 |Q^+|
\end{equation}
Let $0<\varepsilon \leq 1/8$ be arbitrary. Define
\begin{equation}\label{R0def}
R_{\epsilon} =  \frac{{\varepsilon}^{\frac{1}{3}}}{|Q^+| (\norm{v}_{L_{\infty}(Q^{+}(\varepsilon_0))}+1)^{\frac 1 3}} \leq \frac{1}{2}.
\end{equation}
Then \eqref{scaleinvariantu} implies that, for any $0<R\leq R_{\epsilon}$, 
\begin{equation}\label{uest}
C(v,R)^3 \leq \frac{\varepsilon}{2}.
\end{equation}
The assumption (\ref{strongconverguk}) implies that there exists $K_{\varepsilon}: ]0, R_{\epsilon}] \to \N$ such that
\begin{equation}\label{differencessmall}
\frac{1}{R^2} \int_{Q^{+}(1)} |v^{(k)}-v|^3 \, dx \, dt \leq \frac{\varepsilon}{2} \text{ for all } k\geq K_{\varepsilon}(R),
\end{equation}
for all $0<R\leq R_{\epsilon}$.
 From (\ref{uest}) and (\ref{differencessmall}), we have
 \begin{equation}\label{ukest}
 C(v^{(k)},R)^3 = \frac{1}{R^2}\int_{Q^+(R)} |v^{(k)}|^3 \, dx\,dt\leq\varepsilon \text{ for all } k\geq K_{\varepsilon}(R).
 \end{equation}

The next step is to set up the iteration. Using Proposition~\ref{scaleinvariantestimatesflattened} and (\ref{ukest}), we see that, for all $0<\tau<1$, $0 < R \leq R_{\epsilon}$, and $k\geq K_{\epsilon}(R)$, we have
\begin{equation}\label{firstpressureest}
D\left(q^{(k)}, \frac{\tau R}{2}\right) \leq \left[c_{3}\tau^{\frac{4}{3}}+c_4 \tau^{-1} \epsilon^{\frac{1}{3}}\right]D\left(q^{(k)}, R\right)+ c_{5}\tau^{\frac{4}{3}}\left[\epsilon^{\frac{1}{3}}+\epsilon^{\frac{1}{2}}\right]+c_{6}\tau^{-1}\left[\epsilon^{\frac{2}{3}}+\epsilon\right].
\end{equation}
Fix $0 < \tau < 1$ so that $0<\tau<1/(64c_{3}^3)$ (hence, $c_{3}\tau^{\frac{4}{3}}<\tau/4$). With this value of $\tau$, we consider $\epsilon$ such that $0<\epsilon<\tau^6/(4c_{4})^3$ (so that $c_{4} \tau^{-1} \epsilon^{\frac{1}{3}}< \tau/4$).
With these choices, writing $\theta = \tau/2$, and simplifying the RHS of~\eqref{firstpressureest}, we obtain
\begin{equation}\label{secondpressureest}
D\left(q^{(k)}, \theta R \right) \leq \theta D\left(q^{(k)}, R\right)+ \bar{c} \epsilon,
\end{equation}
for all $k\geq K_{\epsilon}(R)$, where $\bar{c}$ depended on our choice of $\tau$.
Using  this and (\ref{ukest}), we see that the following iterative relations hold for all $j \in \N$, $i=0,\ldots,j$, and
$k\geq \max_{i=1,\ldots,j+1}  K_{\varepsilon}\left(\theta^{i} R_{\epsilon}\right)  =: \bar{K}_{\epsilon,j}$.
Namely,
\begin{equation}\label{iterationrelation1}
C\left(v^{(k)}, \theta^{i+1} R_{\epsilon}\right) \leq \varepsilon^{\frac{1}{3}}
\end{equation}
and 
\begin{equation}\label{iterationrelation2}
D\left(q^{(k)}, \theta^{i+1}R_{\epsilon}\right)\leq \theta D\left(q^{(k)}, \theta^{i}R_{\epsilon}\right)+\bar{c} \epsilon.
\end{equation}
Iterating (\ref{iterationrelation2}), it can be inferred that for $i=1,\ldots,j$,
\begin{equation}\label{gradpresiterationperformed}
D\left(q^{(k)}, \theta^{i+1}R_{\epsilon}\right)\leq \theta^{i+1} D\left(q^{(k)}, {R_{\epsilon}}\right)+\frac{\bar{c} \epsilon}{1-\theta},
\end{equation}
 provided that $k\geq \bar{K}_{\epsilon,j}$. The factor $1/(1-\theta)$ comes from summing the geometric series with ratio~$\theta$.
 Now, since $\norm{q^{(k)}}_{L_{\frac{3}{2}}(Q^{+})}\leq M$,
 we have
 \begin{equation}\label{presboundR0}
 D\left(q^{(k)}, { R_{\epsilon}}\right)\leq M'/R_{\epsilon}^{\frac{4}{3}}.
 \end{equation}
Using (\ref{iterationrelation1}), (\ref{gradpresiterationperformed}), and (\ref{presboundR0}), it can be inferred that for $k\geq \bar{K}_{\epsilon,j}$, one has the bound
\begin{equation}\label{seqCKN}
C\left(v^{(k)}, \theta^{j+1}R_{\epsilon}\right)+D\left(q^{(k)}, \theta^{j+1}R_{\epsilon}\right)\leq\varepsilon^{\frac{1}{3}}+\theta^{j+1}M'/R_{\epsilon}^{\frac{4}{3}}+\frac{\bar{c} \epsilon}{1-\theta}.
\end{equation}

To conclude, consider the  following additional constraints on $\varepsilon$. Namely, $\varepsilon\leq\varepsilon_{*}^{3}/8$ (where $\varepsilon_*$ is as in Theorem~\ref{partialregflattened}), and
\begin{equation}
	\frac{\bar{c}\epsilon}{1-\theta} \leq\frac{\varepsilon_*}{4}.
\end{equation}
Since $0<\theta<1/2$, we may fix $j$ sufficiently large such that
\begin{equation}
	\theta^{j+1}M'/R_{\epsilon}^{\frac{4}{3}}\leq\frac{\varepsilon_*}{4}.
\end{equation}
These choices, together with (\ref{seqCKN}), imply that 
\begin{equation}\label{seqCKN1}
C\left(v^{(k)}, \theta^{j+1}R_{\epsilon}\right)+D\left(q^{(k)}, \theta^{j+1}R_{\epsilon}\right)\leq \varepsilon_*.
\end{equation}
 for $k\geq \bar{K}_{\epsilon,j}$. Finally, Theorem~\ref{partialregflattened}~implies that for all $k\geq \bar{K}_{\epsilon,j}$,
\begin{equation}\label{seqregular}
\norm{v^{(k)}}_{L_\infty(Q^+(\bar{R}/2))} \leq \frac{c_*}{\bar{R}},
\end{equation}
where $\bar{R} = \theta^{j+1}R_\epsilon$. Hence, $\limsup_{k \to \infty} \norm{v^{(k)}}_{L_\infty(Q^+(\bar{R}/2))} < \infty$, as desired.
\end{proof}


\section{Parabolic Sobolev embedding}
\label{sec:parabolicsobolev}

In this section, we recall the parabolic Sobolev embedding theorem into H{\"o}lder spaces used in the proof of Proposition~\ref{pro:Holderestimates}. 

\begin{lemma}[Parabolic Sobolev embedding]
\label{lem:parabolicsobolev}
Let $d \geq 1$ be an integer and $\Omega \subset \R^d$ a bounded $C^2$ domain. Let $0 < T \leq \infty$ and $Q_T = \Omega \times ]0,T[$. Suppose that $1 \leq s,l \leq \infty$ satisfy
\begin{equation}
	0 < \alpha := 2-\frac{d}{s}-\frac{2}{l}  \leq 1.
\end{equation}
If $u \in W^{2,1}_{s,l}(Q_T)$, then
\begin{equation}
	\norm{u}_{C^\alpha_\para(Q_T)} \leq C(d,\Omega,T,s,l) \norm{u}_{W^{2,1}_{s,l}(Q_T)}.
\end{equation}
\end{lemma}
In the above, $C^\alpha_\para$ ($0 < \alpha \leq 1$) represents the class of $\alpha$-H{\"o}lder continuous functions in the metric $|x-y| + \sqrt{|t-s|}$ on the relevant domain. 

Lemma~\ref{lem:parabolicsobolev} is well known, with reference often made to the classic book of Ladyzhenskaya, Solonnikov, and Uraltseva~\cite{ladyzhenksayasolonnikovuraltseva}, e.g., Lemma 3.3, p. 80 (which only treats the case $p=q$). A different proof is to apply an extension operator and combine a parabolic Poincar{\'e}-Sobolev inequality with Morrey- and Campanato-type arguments. We omit the details.

If $u$ vanishes on the parabolic boundary, we also have
\begin{corollary}
\label{cor:parabolicsobolevzerobcs}
Assume the hypotheses of Lemma~\ref{lem:parabolicsobolev}. If additionally $u|_{\p'Q_T} = 0$, then
\begin{equation}
	[u]_{C^\alpha_\para(Q_T)} \leq C(d,\Omega,s,l) [u]_{W^{2,1}_{s,l}(Q_T)}.
\end{equation}
The constant is independent of $T$ and translation, rotation, and dilation of $\Omega$.
\end{corollary}
Corollary~\ref{cor:parabolicsobolevzerobcs} was used to prove Corollary~\ref{cor:HolderperturbedNSE}.

\section{Neumann problem for the pressure}
\label{sec:neumannprob}

In this section, we compile known \emph{weighted} estimates for the elliptic problem
\begin{equation}
\label{eq:ellipticproblem}
\left\lbrace
\begin{aligned}
	\Delta p &= 0 & &\text{ in } \Omega \\
	\frac{\p p}{\p n} &= \div_{\p \Omega} An & &\text{ on } \p\Omega,
	\end{aligned}
	\right.
\end{equation}
where $A \: \overline{\Omega} \to \R^{d\times d}$ is an antisymmetric matrix and $\Omega \subset \R^d$ is a bounded $C^2$ domain ($d \geq 2$). This problem appears naturally in the work of Abe and Giga \cite{AbeGiga,abestokesflow} in order to estimate the harmonic pressure and in Kenig, Lin, and Shen's paper~\cite{keniglinshen} in the context of homogenization (though similar weighted estimates may go back farther).

Because $A$ is antisymmetric, $An$ is a tangential vector field on $\p\Omega$. The operator $\nabla_{\p\Omega}$ and its adjoint $\div_{\p\Omega}$ are intrinsically defined on the manifold $\p\Omega$ (with the ambient metric). By extending into the domain $\Omega$, we have the equivalent extrinsic definition $\nabla_{\p\Omega} = \nabla - n \frac{\p}{\p n}$ and $\div_{\p \Omega} = \tr \nabla_{\p \Omega}$. (When $f$ is a vector field, $\nabla_{\p \Omega} f$ is interpreted as a matrix.) For convenience, we impose $A \in H^1(\Omega)$.

We say that $p \in H^1(\Omega)$ ($\fint_{\Omega} p \, dx = 0$) is a \emph{weak solution} of~\eqref{eq:ellipticproblem} if $p$ satisfies
\begin{equation}
	\label{eq:weakformulation}
	\int_\Omega \nabla p \cdot \nabla \varphi \, dx + \int_{\p \Omega} An \cdot \nabla \varphi \, dS = 0
\end{equation}
for all $\varphi \in C^2(\overline{\Omega})$. (Notice that $An \cdot \nabla \varphi = An \cdot \nabla_{\p \Omega} \varphi$.) By density and $A|_{\p \Omega} \in H^{\frac{1}{2}}(\p\Omega)$,~\eqref{eq:weakformulation} will be verified for all $\varphi \in H^1(\Omega)$. Hence, existence and uniqueness is guaranteed by the Riesz representation theorem in $H^1_\avg(\Omega)$, consisting of $H^1(\Omega)$ functions with zero average, with inner product $\la p,q \ra = \int_{\Omega} \nabla p \cdot \nabla q \, dx$.

\begin{lemma}[Elliptic estimate]
\label{lem:ellipticest}
Let $\Omega \subset \R^d$ be a bounded $C^2$ domain ($d \geq 2$) and $A \in H^1(\Omega)$ be an antisymmetric matrix.
	Let $p \in H^1(\Omega)$ ($\fint_\Omega p \, dx = 0$) be the unique weak solution of~\eqref{eq:ellipticproblem}.
Then the following estimates hold:
\begin{itemize}
\item If $A \in C(\overline{\Omega})$, then
	\begin{equation}
	\label{ellipticest1}
	\sup_{x \in \Omega} \dist(x,\Omega) |\nabla p(x)| \leq C(d) \norm{A}_{L_\infty(\p \Omega)}.
	\end{equation}
\item For all $0 < \alpha < 1$, if $A \in C^\alpha(\overline{\Omega})$, then
	\begin{equation}
	\label{ellipticest2}
	\sup_{x \in \Omega} \dist(x,\Omega)^{1-\alpha} |\nabla p(x)| \leq C(d,\alpha) [A]_{C^\alpha(\Omega)}.
	\end{equation}
\end{itemize}
\end{lemma}
The estimate~\eqref{ellipticest1} was proved by Abe and Giga in \cite{AbeGiga,abestokesflow} 
for bounded and exterior $C^3$ domains by a blow-up argument and by Kenig, Lin, and Shen in \cite[Lemma 6.2]{keniglinshen} for bounded $C^{1,\gamma}$ domains by directly estimating the kernel representation. Technically, \cite{keniglinshen} assumes $A \in C^1(\overline{\Omega})$, but one may use an approximation argument to obtain $A \in C(\overline{\Omega})$. The estimate~\eqref{ellipticest2} is proved in~\cite{abestokesflow} for uniformly $C^2$ domains. Hence, bounded $C^2$ domains are \emph{strongly admissible}, in the language of~\cite[Remark 2.10]{abestokesflow}. See~\cite[Section 2]{abestokesflow} for an overview of the history and terminology.

In~\cite{AbeGiga,abestokesflow}, Abe and Giga use a slightly different notion of solution which is adapted to the estimates~\eqref{ellipticest1}-\eqref{ellipticest2}. This is necessary for their blow-up arguments. However, in the context of Lemma~\ref{lem:ellipticest}, their solutions agree with the unique weak solution described above.

We now mention how~\eqref{eq:ellipticproblem} is relevant to the Navier-Stokes equations. In the next two results, we adopt the notation of Lemma~\ref{lem:ellipticest}.
\begin{lemma}
\label{cor:pressurecor1}
If $w \in H^2(\Omega)$ is a divergence-free vector field, then $\pi$ ($\fint_\Omega \pi \, dx = 0$) satisfying
\begin{equation}
	\nabla \pi =\bQ(\Delta w)
\end{equation}
is the unique weak solution of~\eqref{eq:ellipticproblem} with $A = (\nabla w)^T - \nabla w$.
\end{lemma}
This is used to prove Lemma~\ref{lem:subpressurelem}.
Here, $\bP$ and $\bQ$ represent the orthogonal projections onto divergence-free and gradient fields, respectively, in the Helmholtz decomposition.
Since $\Delta w$ and $\bP(\Delta w)$ are divergence free, $\bQ(\Delta w)$ must be divergence free as well. Hence, $\Delta \pi = 0$. Since $\frac{\p \pi}{\p n} = \Delta w \cdot n$, a direct computation (see the proof of Proposition~4.1 in~\cite{abestokesflow}) gives
\begin{equation}
	\Delta w \cdot n = \div_{\p \Omega} ((\nabla w)^T - \nabla w) n).
\end{equation}

\begin{lemma}
\label{cor:pressurecor2}
Let $F \in H^1_0(\Omega;\R^{d\times d})$ be a matrix-valued function and 
\begin{equation}
	\Phi_1 = (-\Delta)^{-1} \div \div F. 
\end{equation}
In other words, $\Phi_1 = \sum_{i,j=1}^d \cR_i \cR_j F_{ij}$, where $\cR_i$ is the $i$th Riesz transform ($1 \leq i \leq d$).
Additionally, write $h = \Gamma \ast \div F$, where $\Gamma$ is the fundamental solution of $-\Delta$. Then $\Phi_2$ ($\fint_\Omega \Phi_2 \, dx = 0$) satisfying
\begin{equation}
	\nabla \Phi_2 = \bQ(\div F) - \nabla \Phi_1
\end{equation}
is the unique weak solution of~\eqref{eq:ellipticproblem} with $A = \nabla h - (\nabla h)^T$.
\end{lemma}

This is proven for smooth, compactly supported $F$ in \cite[Proposition~3.2]{abestokesflow}. We use it with $F = u \otimes u$ to conclude the proof of Proposition~\ref{pro:pressureestimates}.

\begin{remark}[Boundary regularity in Theorem~\ref{thm:mbas}]
\label{rmk:regularityremark}
The current obstruction to $C^2$ regularity of $\Omega$ in Theorem~\ref{thm:mbas} is Proposition~\ref{pro:AbeGiga}.  While it is not recorded in the literature, it appears that Proposition~\ref{pro:AbeGiga} remains true when $\Omega$ is $C^{2,\alpha}$. Indeed, $C^3$ was exploited in~\cite{AbeGiga} in two major ways: i) to prove ``admissibility"~\eqref{ellipticest1}, but the assumption has been weakened to $C^{1,\gamma}$, as mentioned above, and ii) to apply the Schauder theory for the Stokes equations developed by Solonnikov in~\cite{Sol1977}.
\end{remark}

\end{appendix}

\subsubsection*{Acknowledgments}
The authors would like to thank Gregory Seregin and Vladim{\'i}r {\v S}ver{\'a}k for helpful discussions. DA was supported by the US Department of Defense through the NDSEG Graduate Fellowship and by a travel grant from the Council of Graduate Students at the University of Minnesota.

\subsubsection*{Ethics}
The authors declare no conflicts of interest.

\bibliographystyle{plain}
\bibliography{journalsubmissionrevised}

\end{document}